\documentclass[11pt]{article}
\usepackage[dvips]{graphicx}
\usepackage{ amssymb,amsbsy  }
\usepackage{ifthen}
\usepackage{cite}

\usepackage{amsfonts,amsmath,amsthm}
\usepackage[margin=1.8cm]{geometry}
\usepackage[algo2e,linesnumbered,vlined,ruled]{algorithm2e}
\usepackage{xcolor}
\usepackage{verbatim}
\usepackage{dsfont}
\usepackage{hyperref}
\usepackage{fancyhdr}
\usepackage{cancel}
\usepackage{soul}
\usepackage[page,title]{appendix}
\newtheorem{theorem}{Theorem}[section]
\newtheorem{lemma}[theorem]{Lemma}

\theoremstyle{definition}
\newtheorem{definition}[theorem]{Definition}
\newtheorem{example}[theorem]{Example}

\newtheorem{proposition}[theorem]{Proposition}
\newtheorem{corollary}[theorem]{Corollary}
\theoremstyle{remark}

\numberwithin{equation}{section}
\usepackage{setspace,url}
\usepackage{graphics,graphicx,epstopdf}
\usepackage{tabularx}
\usepackage{longtable}
\usepackage{booktabs,multirow,array,multicol}
\usepackage{enumitem}
\usepackage{subfig}

\ifpdf
  \DeclareGraphicsExtensions{.eps,.pdf,.png,.jpg}
\else
  \DeclareGraphicsExtensions{.eps}
\fi
\graphicspath{ {./results/} }

\newcommand{\be}{\begin{equation}}
\newcommand{\ee}{\end{equation}}

\title{Second-Order Optimality Conditions for Nonsmooth Constrained Optimization with Applications to Bilevel Programming}
\author{Xiang Liu\thanks{\baselineskip 9pt School of Mathematical Sciences, Dalian University of Technology, Dalian 116024, China. E-mail: liuxiang@mail.dlut.edu.cn}, \
Mengwei Xu\thanks{\baselineskip 9pt Institute of Mathematics, Hebei University of Technology, Tianjin 300401, China.
E-mail: xumengw@hotmail.com. The research of this author is supported by [No. 12071342, the National Natural
Science Foundation of China], [Nos. A2020202030 and A2025202027, the Natural Science Foundation of Hebei Province] and [No. J20230701, the Open ProjectProgram of Key Laboratory of
Discrete Mathematics with Applications of Ministry of Education, Fuzhou University].}  \   \ and \ Liwei Zhang\thanks{\baselineskip 9pt
School of Mathematical Sciences, Dalian University of Technology, Dalian 116024, China. E-mail: lwzhang@dlut.edu.cn. The
research of this author was supported by [No. 2022YFA1004000, the National Key R\&D Program of China] and [No. 12371298, the National Natural Science Foundation of China]. }
}
\begin{document}
\maketitle
\begin{abstract}
Second-order optimality conditions are essential for nonsmooth optimization, where both the objective and constraint functions are Lipschitz continuous and second-order directionally differentiable. This paper provides no-gap second-order necessary and sufficient optimality conditions for such problems without requiring convexity assumptions on the constraint set. We introduce the concept of second-order gph-regularity for constraint functions, which ensures the outer second-order regularity of the feasible region and enables the formulation of comprehensive optimality conditions through the parabolic curve approach.
An important application of our results is bilevel optimization, where we derive second-order necessary and sufficient optimality conditions for bi-local optimal solutions, which are based on the
local solutions of the lower-level problem.
By leveraging  the  Mangasarian-Fromovitz constraint qualification (MFCQ),   strong second-order sufficient condition (SSOSC) and constant rank constraint
qualification (CRCQ) of lower-level problem,
these second-order conditions  are derived without  requiring the uniqueness of the lower-level
multipliers.
In addition, if the linear independence constraint qualification (LICQ) holds, these conditions are expressed solely in terms of the second-order derivatives
 of the functions defining the bilevel problem, without relying on the second-order information
 from the solution mapping, which would introduce implicit complexities.
\end{abstract}



\newpage

\baselineskip 18pt
\parskip 2pt
\section{Introduction}

Second-order optimality conditions are of paramount importance in optimization theory and numerical analysis, providing fundamental criteria for identifying local minimizers and guiding the convergence of advanced numerical algorithms. For nonlinear programming with smooth data, a comprehensive framework of no-gap second-order necessary and sufficient conditions was well established  \cite{Ioffe 79,Ben-Tal 80,Ben-Tal 82}.
Nevertheless, many real-world problems involve functions that are only Lipschitz continuous and possess only second-order directional differentiability, making classical smooth analysis inadequate for capturing  local behavior.

In this paper, we investigate constrained nonsmooth optimization problems of the form
\begin{align*}
(P)~~~~\min &~~~~ f(x)\\
{\rm s.t.}&~~~~ G(x)\in K,
\end{align*}
where $f:\mathbb{R}^n\to \mathbb{R}$ and $G:\mathbb{R}^n\to \mathbb{R}^m$ are Lipschitz continuous and second-order directionally differentiable, $K\subseteq \mathbb{R}^m$ is a closed set.
The feasible set is defined by $\Psi:=\{x\in \mathbb{R}^n:G(x)\in K\}$.
The central challenge in this nonsmooth context is to derive second-order optimality conditions without imposing overly stringent regularity assumptions, such as the convexity of  $K$.

A substantial body of literature  addressed second-order optimality conditions under smooth or convex settings. In the case of convex constraint sets, the parabolic curve approach  and convex duality theory were   effectively applied, with an additional sigma term introduced to refine the analytical framework \cite{Bonnans,Bonnans 99,Cominetti 90,Shapiro 97}.
Recent advances have further extended these results to nonconvex settings; for example, Gfrerer et al. \cite{Gfrerer Ye and Zhou 22} developed both primal and dual second-order necessary conditions using second-order tangent sets and convex duality, while Mordukhovich \cite{Mordukhovich 21} replaced the sigma term with the second-order subderivative of the indicator function on $K$.
Moreover, Benko et al. \cite{Benko Gfrerer Zhou 23} relaxed classical assumptions by eliminating the need for convexity and parabolic derivability of $K$.
 
In the context of nonsmooth optimization, techniques based on second-order subderivatives and parabolic second-order epiderivatives were successfully employed for problems with nonsmooth objectives but smooth constraints (see, e.g., \cite[Section 3.3.5]{Bonnans}, \cite{Cominetti 91}, \cite{Rockafellar 89} and \cite[Chapter 13]{Rockafellar}).
Nevertheless, fully nonsmooth problems—where both the objective and the constraints are nonsmooth—remain less explored, with notable contributions addressing the special case $K=\mathbb{R}^m_-$ \cite{Ruckmann}.

In this work, we introduce the novel concept of second-order gph-regularity  for constraint functions—a critical tool that preserves the outer second-order regularity of the feasible set in nonsmooth settings.
Many simple functions, involving the absolute function, the minimal function and  piecewise $C^2$ functions, satisfy
the second-order gph-regularity.
Moreover, the locally Lipschitz inverse mapping maintains this property  under the  maximal rank condition of its subdifferential.
Using the parabolic curve approach, we establish comprehensive no-gap second-order necessary and sufficient optimality conditions for problem $(P)$ under the metric subregularity constraint qualification (MSCQ).

 An important application of our theoretical contributions lies in bilevel optimization—a class of problems characterized by a hierarchical structure and widely encountered in fields ranging from economics to engineering (see, e.g.,  \cite{Dempe 02,Dempe 06,Dempe 15, Dempe 20, Kunapuli 08}).
When the lower-level problem has a unique solution, traditional bilevel formulations typically employ the classical implicit function reformulation, which substitutes the global solution mapping of the lower-level problem into the upper-level model \cite{Falk 95,Dempe 92,  Mehlitz 17, Outrata, Shimizu 97, Zaslavski 12}.
In contrast, our approach focuses on the bi-local solution, where only the local solution mapping of the lower-level problem is incorporated into the upper-level formulation \cite{Liu 25}. This distinction is critical: while the classical approach  uses the entire global solution set of the lower-level problem, our bi-local reformulation (denoted as (SP)) captures more refined local behavior, leading to optimality conditions under milder assumptions.
Notably, when the lower-level problem is convex, the bi-local solution reduces to the classical local optimal solution.

Our analysis proceeds by establishing the second-order gph-regularity of the local solution mapping of a parametric problem and rigorously examining its second-order directional differentiability.
We then   derive both first- and second-order optimality conditions for bi-local
solutions  by the implicit function reformulation without the uniqueness assumption  of the lower-level multipliers.
Significantly,   under the SSOSC and LICQ,  these conditions are equivalent to those obtained from the classical first-order reformulation,  where the lower-level solution set is replaced by its first-order necessary conditions. This equivalence underscores the robustness of our approach, as the derived optimality conditions depend solely on the second-order information inherent in the functions defining the bilevel problem and do not require second-order data from the local solution mapping--which would introduce implicit complexities.

The remainder of this paper is organized as follows. Section \ref{section 2} introduces the basic notation and key concepts in variational analysis, including second-order gph-regularity. Section \ref{section 3} is devoted to deriving the first- and second-order optimality conditions for the nonsmooth constrained problem $(P)$.
In Section \ref{section 4}, we
demonstrate the  second-order gph-regularity of the solution mapping of a   parametric problem
Then we derive optimality conditions for bi-local solutions by the implicit function reformulation under the MFCQ, SSOSC and CRCQ,  and show the equivalence between the implicit function reformulation and the first-order reformulation under the SSOSC and LICQ.
Finally, Section \ref{section 5}  concludes the paper with a summary of our findings.

\section{Preliminaries and preliminary results}\label{section 2}
We first give notation that will be used in the paper.
Let $[n]$ denote the index set $\{1,2,\cdots,n\}$.
Denote by $\|\cdot\|$ the  $l_2$-norm of a vector $x$ and by $I_{n}$ the $n\times n$
identity matrix.  Let $\mathbb{R}^n_{-}$ denote the nonpositive orthant in $\mathbb{R}^n$.
 For a vector $x$, $x_i$ denotes the i-th component of  $x$ and we denote
 $\boldsymbol{B}_{\delta}(x)=\{x':\parallel x'-x\parallel\le \delta\}$.
  For a convex set $D\subseteq
 \mathbb{R}^n$, we denote by $\Pi_D(w)$ the projection of $w$ onto $D$ and define $d(z,D)=\inf_{y\in D}\|z-y\|$ as
 the distance from $z$ to $D$. We denote by $coD$ the convex
hull of $D$. For vectors $u,v\in \mathbb{R}^n$, we denote by $\langle u,v\rangle$ the inner product and denote
$u^Tv=0$ by $u\bot v$.
The notation $\psi (t)=o(t)$
means  $\psi(t)/t\to0$ as $t\downarrow 0$.
   For a set-valued mapping $S:\mathbb{R}^n \rightrightarrows\mathbb{R}^m $,
  we denote the graph of $S$ by
${\rm gph} S=\{(x,y)\in \mathbb{R}^n \times \mathbb{R}^m: y\in S(x)\}$.
 For a function $f :\mathbb{R}^n \times \mathbb{R}^m \to \mathbb{R} $, a mapping $g :\mathbb{R}^n \times
 \mathbb{R}^m \to \mathbb{R}^q $ and a mapping $y:\mathbb{R}^n\to \mathbb{R}^m$,
 for a fixed point $(x,y)$ with $y=y(x)$, we use $\nabla_x f(x,y(x))$ to denote $\nabla_x f(x,y)|_{y=y(x)}$;
 $\nabla_y f(x,y(x))$ to denote $\nabla_y f(x,y)|_{y=y(x)}$; $\nabla_{xx}^2 f(x,y(x))$ to denote $\nabla_{xx}^2
 f(x,y)|_{y=y(x)}$;  $\nabla_{xy}^2 f(x,y(x))$ to denote $\nabla_{xy}^2 f(x,y)|_{y=y(x)}$;
 $ \mathcal{J}_x g(x,y(x))$ to denote $\mathcal{J}_x g(x,y)|_{y=y(x)}$ and $ \mathcal{J}_y g(x,y(x))$ to denote
 $\mathcal{J}_y g(x,y)|_{y=y(x)}$.
  Let $\mathbb{N}$ denote the set of positive integers. For  $r\in \mathbb{N}$, we use $C^r$ to denote the class of mappings that are
$r$-times continuously differentiable.

We now present some background material on variational analysis which will be used throughout the paper.
Detailed discussions on these subjects can be found in \cite{Bonnans,Rockafellar}.

\begin{definition}[Tangent cones]
   Given $C\subseteq \mathbb{R}^n$ and $x^*\in C$, the tangent cone to $C$ at $x^*$ is defined by
$$\mathcal{T}_{C}(x^*):=\{d\in \mathbb{R}^n:\exists t_k\downarrow 0, d^k\to d~ \text{with}~x^*+t_kd^k\in C \}.$$
For $x^*\in C$ and $d\in \mathcal{T}_{C}(x^*)$, the outer second-order tangent set to $C$ at $x^*$ in direction
$d$ is defined by
$$\mathcal{T}_{C}^2(x^*;d):=\left\{w\in \mathbb{R}^n:\exists t_k\downarrow 0, w^k\to w~
\text{with}~x^*+t_kd+\frac{1}{2}t_k^2w^k\in C \right\}.$$
\end{definition}

 To derive the second-order sufficient optimality conditions for the nonsmooth problem with set constraint, we recall the outer second-order regularity \cite[Definition 3.85]{Bonnans}, which are held for polyhedral sets and cones of
positive semidefinite matrices.

\begin{definition}[Outer second-order regularity]
 Let $K$ be a closed subset of $\mathbb{R}^m$,  $y^*\in K$ and $d\in\mathcal{T}_{K}(y^*)$. We say that $K$ is outer second-order regular at $y^*$ in direction $d$, if for any sequence $y^k\in K$ of the form
$y^k := y^* + t_kd +\frac{1}{2}t_k^2w^k$, where $t_k\downarrow0$ and $\{w^k\}$ is a sequence satisfying
$t_kw^k\to 0$, the following condition holds
\begin{align}\label{}
  &\lim_{k\to \infty} d\left(w^k,\mathcal{T}_{K}^2(y^*;d)\right)=0.\notag
\end{align}
\end{definition}
%{\color{blue}For example, polyhedral sets and cones of positive semidefinite matrices are outer second order regular from \cite{Bonnans}.}
%A mapping $g:\mathbb{R}^n\to \mathbb{R}^m$ is said to be locally Lipschitz continuous at $x^*$,
%modulus $c(c\ge 0)$, if there exists a neighborhood $U$ of $x^*$ for all $x_1,x_2\in U$, it follows that
%$$\|g(x_1)-g(x_2)\|\le c\|x_1-x_2\|.$$
%Let $G:\mathbb{R}^n\to \mathbb{R}^m$ be a locally Lipschitz continuous mapping over an open set $\mathcal{O}$,
%then $G$ is differentiable almost everywhere in $\mathcal{O}$. Let $\mathcal{D}_G$ denote the set of
%differentiable points of $G$ in $\mathcal{O}$, for the point $x\in \mathcal{O}$, the B-differential of $G$ at $x$
%is defined by
%\begin{equation*}
%  \partial_B G(x)=\left\{V\in \mathbb{R}^{m\times n}:\exists x^k\in \mathcal{D}_G, \ x^k\to x,\
%  \mathcal{J}G(x^k)\to V\right\},
%\end{equation*}
%and the Clarke subdifferential of $G$ at x is defined by
%\begin{equation*}
%  \partial G(x)=co \partial_B G(x).
%\end{equation*}
%For differential properties of Lipschitz mappings, see the famous book \cite{Clarke}.

\begin{definition}[Subdifferentials]
Let $g: \mathbb{R}^n \rightarrow \mathbb{R}^m$ be a locally Lipschitz continuous mapping  over an open set $\mathbb{R}^n$ and $x^*\in \mathbb{R}^n$.
We define the Bouligand subdifferential (B-subdifferential) of $g$ at $x^*$
\begin{equation*}
  \partial_B g(x^*)=\left\{V\in \mathbb{R}^{n\times m}:\exists x^k\in \mathcal{D}_g, \ x^k\to x^*,\
  \mathcal{J} g(x^k)\to V\right\},
\end{equation*}
where $\mathcal{D}_g$ denotes the set of differentiable points of $g$ in $\mathbb{R}^n$.
% We define the Fr\'{e}chet/regular  subdifferential
%(\cite[Definition 8.3]{Rockafellar})
% of $f $ at $  x$
%as
%\begin{eqnarray*}
%\widehat{\partial} f( x)
%:=
%\{\zeta \in \mathbb{R}^d:  f(x') - f( x) -\langle \zeta, x'-x \rangle  \geq o\|x'- x\| \},
%\end{eqnarray*}
%and the limiting/Mordukhovich/basic  subdifferential of $f$ at ${x}$
%%to be the set
%as
%\begin{eqnarray*}
%\partial^L f({x}):=
%%\{\zeta \in \mathbb{R}^d: (\zeta,-1)\in \mathcal{N}_{{\rm epi} f}(\bar x, f(\bar x)\}\\
%\{\lim_{k\rightarrow \infty} \xi_k: \xi_k \in \widehat{\partial}f(x_{k}), x_{k}\rightarrow  x,f(x_{k})\rightarrow f(x)\}.
%\end{eqnarray*}
The  Clarke generalized Jacobian of $g$ at $x^*$ is
$$ \partial g(x^*):={co} \partial_B g(x^*).$$
%Let $f: \mathbb{R}^d \rightarrow  \mathbb{R}$ be Lipschitz continuous at $ x$.
%We say that $f$ is subdifferentially/Clarke regular at $ x$ provided  that $\partial f({x})=\widehat{\partial} f(  x)$ \cite[Corollary 8.11]{Rockafellar}.
\end{definition}

%\subsection{Directional differentiability}
%The following definitions for directional differentiability can be found in \cite[Section 2.2]{Bonnans}.
%Consider a mapping $g:\mathbb{R}^n\to \mathbb{R}^m$.
\begin{definition}[Directional derivatives]
   We say that $g:\mathbb{R}^n\to \mathbb{R}^m$ is directionally differentiable at a point $x^*$ in
direction $d$ if the limit
$$g'(x^*;d):=\lim_{t\downarrow 0}\frac{g(x^*+td)-g(x^*)}{t}$$
exists. If $g$ is directionally differentiable at $x^*$ in every direction $d\in\mathbb{R}^n $, we say that
$g$ is directionally differentiable at $x^*$. Moreover, if
$$g'(x^*;d)=\lim_{\substack{t\downarrow 0\\ d'\to d}}\frac{g(x^*+td')-g(x^*)}{t},$$
then we say that $g$ is directionally differentiable at $x^*$ in the Hadamard
sense.
\end{definition}
\begin{definition}[Second-order directional derivatives]\label{2-der}
  We say that $g$ is second-order directionally differentiable at a point $x^*$ in
direction $d$ if $g'(x^*;d)$ and the limit
$$g''(x^*;d,w):=\lim_{t\downarrow 0}\frac{g(x^*+td+\frac{1}{2}t^2w)-g(x^*)-tg'(x^*;d)}{\frac{1}{2}t^2}$$
exists for all $w\in \mathbb{R}^n$. If $g$ is second-order directionally differentiable at $x^*$ in every direction
$d\in\mathbb{R}^n $, we say that
$g$ is second-order directionally differentiable at $x^*$.
Moreover, if the limit
$$g''(x^*;d,w)=\lim_{\substack{t\downarrow 0\\ w'\to
w}}\frac{g(x^*+td+\frac{1}{2}t^2w')-g(x^*)-tg'(x^*;d)}{\frac{1}{2}t^2}$$
exists for all $w\in \mathbb{R}^n$,
then we say that $g$ is second-order directionally differentiable at $x^*$ in
direction $d$, in the sense of Hadamard.
\end{definition}
When $g$ is second-order continuously differentiable at $x^*$, then
\begin{align}\label{C2 g'= g''=}
  g'(x^*;d)&=\mathcal{J}g(x^*)d \quad\text{and}\quad
  g''(x^*;d,w)=d^T\nabla^2g(x^*)d+\mathcal{J}g(x^*)w,
\end{align}
where $d^T\nabla^2g(x^*)d:=\left[d^T\nabla^2g_1(x^*)d,\cdots,d^T\nabla^2g_m(x^*)d\right]^T$.

 It is known from \cite[Proposition 2.49]{Bonnans} that if $g$ is directionally differentiable and locally Lipschitz continuous with  modulus $c$ at $x^*\in\mathbb{R}^n$, then $g$ is directionally
differentiable at $x^*$ in the Hadamard sense and the directional derivative $g'(x^*;\cdot)$ is
Lipschitz continuous with modulus $c$ on $\mathbb{R}^n$.
We extend this result to the second-order directional differentiability.

\begin{proposition}\label{prop lipschitz--Hadamard directional differentiable}
  Suppose that $g:\mathbb{R}^n\to \mathbb{R}^m$ is second-order directionally
differentiable at $x^*$ in direction $d$ and locally Lipschitz continuous
 with  modulus $c$ at $x^*\in\mathbb{R}^n$, then $g$ is second-order directionally differentiable at $x^*$ in
direction $d$, in the Hadamard sense and
the second-order directional derivative $g''(x^*;d,\cdot)$ is Lipschitz continuous with modulus $c$ on $\mathbb{R}^n$.
\end{proposition}
\begin{proof}
From the definition of second-order directional derivatives, we have that
  \begin{align*}
    &\left\|g''(x^*;d,w)-g''(x^*;d,v)\right\|\\
=&\left\|\lim_{t\downarrow 0}\frac{g(x^*+td
+\frac{1}{2}t^2w)-g(x^*)-tg'(x^*;d)}{\frac{1}{2}t^2}-\lim_{t\downarrow
0}\frac{g(x^*+td+\frac{1}{2}t^2v)-g(x^*)-tg'(x^*;d)}{\frac{1}{2}t^2}\right\|\\
= &\left\|\lim_{t\downarrow0}\frac{g(x^*+td+\frac{1}{2}t^2w)-g(x^*+td+\frac{1}{2}t^2v)}{\frac{1}{2}t^2}\right\|\\
\leq& \lim_{t\downarrow
0}\frac{\left\|g(x^*+td+\frac{1}{2}t^2w)-g(x^*+td+\frac{1}{2}t^2v)\right\|}{\frac{1}{2}t^2}\\
\le& c\|w-v\|,
  \end{align*}
  for all $w,v\in \mathbb{R}^n$.
 Thus $g''(x^*;d,\cdot)$ is Lipschitz continuous with modulus $c$ on $\mathbb{R}^n$, then
 it is also second-order
 directionally differentiable at $x^*$ in the direction $d$, in the Hadamard sense from its definition.
\end{proof}

In order to study sufficient optimality conditions  via the parabolic curve approach, an additional regularity condition is indispensable.
The second-order epi-regularity was presented in \cite[Definition 3.94]{Bonnans}, which was applied to the case of nonsmooth problems with inequality constraints in \cite[Proposition 2.1]{Ruckmann}.

 \begin{definition}\label{def Second-order epi-regular}[Second-order epi-regularity]
Let $g:\mathbb{R}^n\to \mathbb{R}^m$ be locally Lipschitz continuous and second-order directionally differentiable at a point $x^*$.
   We say that $g$ is second-order epi-regular at $x^*$ in
  direction $d$, if the following condition holds: for
every path $w:\mathbb{R}_{+}\to \mathbb{R}^n$ satisfying $tw(t)\to 0$ as $t\downarrow 0$, there exists  a mapping $r(x):\mathbb{R}_{+}\to \mathbb{R}^m$ such that
$r(t^2)=o(t^2)$ and
\begin{align}
  & g\left(x^*+td+\frac{1}{2}t^2w(t)\right)\ge g(x^*)+tg'(x^*;d)+\frac{1}{2}t^2g''(x^*;d,w(t))+r(t^2)\notag
\end{align}
is satisfied for $t\ge 0$. We say that $g$ is second-order epi-regular at $x^*$
if it
is second-order epi-regular at $x^*$ for every direction $d\in\mathbb{R}^n$.
\end{definition}

We introduce the concept of second-order gph-regularity, which preserves the outer second-order regularity of the feasible set for the general set constraint.
\begin{definition}\label{def Second-order gph-regular}[Second-order gph-regularity]
 Let $g:\mathbb{R}^n\to \mathbb{R}^m$ be locally Lipschitz continuous and second-order directionally differentiable at a point $x^*$.
   We say that $g$ is second-order gph-regular at $x^*$ in
  direction $d$, if the following condition holds: for
every path $w:\mathbb{R}_{+}\to \mathbb{R}^n$ satisfying $tw(t)\to 0$ as $t\downarrow 0$, there exists  a mapping $r(x):\mathbb{R}_{+}\to \mathbb{R}^m$ such that
$r(t^2)=o(t^2)$ and
\begin{align}\label{align second-order gph-regular}
  & g\left(x^*+td+\frac{1}{2}t^2w(t)\right)=g(x^*)+tg'(x^*;d)+\frac{1}{2}t^2g''(x^*;d,w(t))+r(t^2)
\end{align}
is satisfied for $t\ge 0$. We say that $g$ is second-order gph-regular at $x^*$
if it
is second-order gph-regular at $x^*$ for every direction $d\in\mathbb{R}^n$.
\end{definition}
It is clear that if $g$ is  second-order
gph-regular, then it is also second-order
epi-regular. Additionally, a twice continuously differentiable mapping  is  second-order
gph-regular and second-order
epi-regular.

The second-order gph-regularity is crucial for investigating the optimality conditions of constrained nonsmooth problems.
The second-order gph-regularity is a common property satisfied by many functions. We now provide some examples.
\begin{example}\label{abs}
  Consider an absolute value function $f(x)=|x|$, it is easy to see that
  \begin{align}%\label{metric projection Pi' and Pi''}
     & f'(x;d)=\left\{\begin{array}{ll}
                                        d, &x>0  \\
                                         |d|,&x=0\\
                                        -d,&x<0
                                       \end{array}\right.\text{and}\quad
     f''(x;d,w)=\left\{\begin{array}{ll}
                                         w,&x>0 ~\text{or}~ x=0,d>0 \\
                                        |w|,&  x=0,d=0\\
                                         -w,& x<0 ~\text{or}~ x=0,d<0
                                       \end{array},\right.\notag
  \end{align}
   then for all small enough $t\ge0$, it holds that $f(x+td)=f(x)+tf'(x;d)$ and
$$
f(x+td+\frac{1}{2}t^2 w) =f(x)+tf'(x;d)
   +\frac{1}{2}t^2f''(x;d,w),
$$
then  (\ref{align second-order gph-regular}) is easy to check. Thus  the absolute value function
is  second-order gph-regular. Furthermore $l_1$-norm $\|x\|_1=|x_1|+\cdots+|x_n|$ is also second-order gph-regular.
\end{example}

\begin{example}
  Consider the $l_2$-norm $f(x)=\|x\|$, it is a twice continuously differentiable function at
$x\neq 0$. Note that
\begin{align}%\label{metric projection Pi' and Pi''}
     & f'(0;d)=\|d\|\quad\text{and}\quad
     f''(0;d,w)=\left\{\begin{array}{ll}
                                        \|w\|,& d=0 \\
                                       \frac{\langle d,w\rangle}{\|d\|},&  d\neq0
                                       \end{array},\right.\notag
  \end{align}
   then for every path $w:\mathbb{R}_{+}\to \mathbb{R}^n$ satisfying $tw(t)\to 0$ as $t\downarrow 0$,   (\ref{align second-order gph-regular}) holds for $d=0$. When $d\neq 0$, we have
 \begin{align*}
    f\left(0+td+\frac{1}{2}t^2w(t)\right)=& \left\|td+\frac{1}{2}t^2w(t)\right\|=t\left\|d+\frac{1}{2}tw(t)\right\|
    =t\left(\|d\|+\frac{\langle d,\frac{1}{2}tw(t)\rangle}{\|d\|}+o(t)\right)\\
    =&f(0)+tf'(0;d)+\frac{1}{2}t^2f''(0;d,w(t))+o(t^2).
 \end{align*}
   Therefore $l_2$-norm satifies  second-order gph-regularity.
\end{example}

It is known that the minimum function is nonconvex. The following example shows that it is  second-order gph-regular at every point.
\begin{example}\label{example min}
Consider the  minimum function  $f(x)=\min\{x_1,\cdots,x_n\}$, it is easy to see that
  \begin{align}\label{min fun f' and f''}
     & f'(x;d)=\min\{d_i: i\in I(x)\}\quad\text{and}\quad
     f''(x;d,w)=\min\{w_i: i\in I(x,d)\},
  \end{align}
where $I(x):=\{i\in[n]: x_i=f(x)\}$ and $I(x,d):=\{i\in I(x): d_i=f'(x;d)\}$.     Then for all small enough $t\ge0$, it holds that $f(x+td)=f(x)+tf'(x;d)$ and
$$f\left(x+td+\frac{1}{2}t^2 w\right) =f(x)+tf'(x;d)  +\frac{1}{2}t^2f''(x;d,w).$$
Therefore  the  minimum function
is second-order
directionally differentiable
 at every point from $\mathbb{R}^n$.

For every path $w:\mathbb{R}_{+}\to \mathbb{R}^n$ satisfying $tw(t)\to 0$ as $t\downarrow 0$,
then all small enough $t\ge0$, it holds that
\begin{align*}
  f\left(x+td+\frac{1}{2}t^2 w(t)\right) &=\min_{i\in I(x,d)}\left\{x_i+td_i+\frac{1}{2}t^2 w_i(t)\right\}\\
  &= f(x)+tf'(x;d)+\frac{1}{2}t^2\min_{i\in I(x,d)}\{w_i(t)\} \\
  &=f(x)+tf'(x;d)  +\frac{1}{2}t^2f''(x;d,w(t)),
\end{align*}
where the last inequality holds by (\ref{min fun f' and f''}).
Therefore  the  minimum function
is  second-order gph-regular at every point from $\mathbb{R}^n$. Specially, $f(x)=\Pi_{\mathbb{R}_{-}}(x)=\min\{0,x\}$ is second-order gph-regular at every point from $\mathbb{R}^n$.
\end{example}

We now recall the definition of continuous selection \cite[Page 3]{Kuntz 94}.
  Let $r\in \mathbb{N}$, $U$ be an open subset of $\mathbb{R}^n$ and  $g:U \to \mathbb{R}^m$ be a continuous mapping. Then $g$  is called a continuous selection of $C^r$ mappings if there exists a finite number of $C^r$ mappings $g_i:U \to \mathbb{R}^m,~ i\in[l]$, such that the active index set
$I(x)=\{i\in [l]:g(x)=g_i(x)\}$
is nonempty for every $x\in U$.
%The mapping $g_i$ is called active at $x$ if $i\in I(x)$.
The mapping $g$ is called   piecewise $C^r$ (denoted as $PC^r$)  if at every point $x\in U$ there exists a neighborhood $V\subset U$ such that the restriction of $g$ to $V$ is a continuous selection of $C^r$ mappings.

 In the following proposition, we show that $PC^2$ mappings are second-order directional  differentiable and second-order  gph-regular.

\begin{proposition}\label{prop PC2}%\label{PC2}
Let  $U$ be an open subset of $\mathbb{R}^n$. Then a $PC^2$ mapping $g: U\to \mathbb{R}^m$  is second-order
directionally differentiable
and second-order gph-regular at every $x\in U$.
\end{proposition}

\begin{proof}
(i) We first show the second-order directional differentiability of the $PC^2$ mapping.
For any given point $x^*\in U$ and any given  vectors $d,w\in \mathbb{R}^n$, if $i\notin I(x^*)$, i.e., $g(x^*)\neq g_i(x^*)$, then there exists $\epsilon>0$
such that for $t\in(-\epsilon,\epsilon)$, $g\left(x^*+td+\frac{1}{2}t^2w\right)\neq g_i\left(x^*+td+\frac{1}{2}t^2w\right)$  by continuity.

 Due to the directional differentiability  and local Lipschitz continuity of $PC^2$ mappings,
 for any mapping $r(t):\mathbb{R}\to \mathbb{R}^n$ such that
$r(t)=o(t)$  and for any $i\in I(x^*)$,  it holds that
\begin{align}\label{g'-gi'}
   &g'(x^*;d)-g'_i(x^*;d)\notag\\
= &\lim_{t\downarrow 0}\frac{g\left(x^*+td+r(t)\right)-g(x^*)}
{t}-\frac{g_i\left(x^*+td+r(t)\right)-g_i(x^*)}
{t} \notag\\
=&\lim_{t\downarrow 0}\frac{g\left(x^*+td+r(t)\right)-g_i\left(x^*+td+r(t)\right)}
{t}.
\end{align}
 Let $r(t)=\frac{1}{2}t^2w$ and define $ I(x^*,d):=\{i\in I(x^*):g'(x^*;d)=g'_i(x^*;d)\}$. For any $i\in I(x^*)\backslash I(x^*,d)$, it holds that   $g'(x^*;d)\neq g'_i(x^*;d)$. Then by (\ref{g'-gi'}), there exists  $\epsilon_1<\epsilon$ such that
 $g\left(x^*+td+\frac{1}{2}t^2w\right)\neq g_i\left(x^*+td+\frac{1}{2}t^2w\right)$ holds when $t\in (-\epsilon_1,\epsilon_1)$.
Thus $g\left(x^*+td+\frac{1}{2}t^2w\right)\in \{g_{i}\left(x^*+td+\frac{1}{2}t^2w\right):i\in I(x^*,d)\}$.

For $t\in (-\epsilon_1,\epsilon_1)$, denote by $\varphi_i(t):=g_i\left(x^*+td+\frac{1}{2}t^2w\right)-g_i(x^*)-tg'_i(x^*;d)$ for $i\in I(x^*,d)$ and
  $$\varphi(t):=g\left(x^*+td+\frac{1}{2}t^2w\right)-g(x^*)-tg'(x^*;d).$$
The function $\varphi$
    is  a continuous selection of the $C^2$ mappings $\varphi_i$, $i\in I(x^*,d)$.
Obviously $\varphi_i(0)=0$ and $\varphi_i'(0)=0$.

 If second-order  derivatives $\varphi''_i(0)$ and $\varphi''_j(0)$ are not equal, then from the Taylor expansion theorem,
 we derive that
 \begin{align*}
    \lim_{t\downarrow 0} \frac{\varphi_i(t)-\varphi_j(t)}{1/2t^2}
    =& \lim_{t\downarrow 0} \frac{\varphi_i(t)-\varphi_i(0)-t\varphi'_i(0)}{1/2t^2}
    -\frac{\varphi_j(t)-\varphi_j(0)-t \varphi'_j(0)}{1/2t^2}\\
    =&\varphi''_i(0)-\varphi''_j(0)\neq 0.
 \end{align*}
   Hence for every sufficiently small $t>0$, it holds that $\varphi_i(t)\neq \varphi_j(t)$ and then $\varphi(t)$ achieves at most one of the values $\varphi_i(t)$ or $\varphi_j(t)$.
 Similarly as \cite[Proposition 2.1]{Kuntz 94}, we obtain that there exists $\epsilon_2<\epsilon_1$ and an index set $I(x^*,d,w)\subset I(x^*,d)$ such that for $t\in (-\epsilon_2,\epsilon_2)$,
the mapping $\varphi$ is a continuous selection of the mappings $\varphi_i, i\in I(x^*,d,w)$ and all  second-order derivatives $\varphi''_i(0), i\in I(x^*,d,w)$, coincide.
 From the fact that $I(x^*,d,w)\subseteq I(x^*,d)$, it follows that $g$ is a continuous selection of the mappings $g_i, i\in I(x^*,d,w)$.

Consequently, for every $i\in I(x^*,d,w)\subset I(x^*,d)$, we have that
 the following limit  exists and coincides with the common second-order derivative of the mappings $\varphi_i$,
\begin{align}\label{varphi''=}
   & \lim_{t\downarrow 0}\frac{\varphi(t)}{1/2t^2}=
\lim\limits_{t\downarrow 0}\frac{\varphi_i(t)-\varphi_i(0)-t\varphi_i'(0)}{1/2t^2}=\varphi''_i(0).
\end{align}
Otherwise, there exists a sequence $\{t_k\}_{k\in \mathbb{N}}$ with $t_k\downarrow 0$ such that $\lim\limits_{k \to \infty}\frac{\varphi(t_k)}{1/2t_k^2}\neq \varphi''_i(0), i\in I(x^*,d,w)$. Since the index set $I(x^*,d,w)$ is finite, then there exists a subsequence $\{t_{k_l}\}_{l\in \mathbb{N}}$ of $\{t_k\}_{k\in \mathbb{N}}$ and  some index $i\in I(x^*,d,w)$ such that $\varphi(t_{k_l})=\varphi_i(t_{k_l})$. Then using $\varphi_i(0)=0$ and $\varphi_i'(0)=0$, we derive that
$$\lim\limits_{l \to \infty}\frac{\varphi(t_{k_l})}{1/2t_{k_l}^2}=
\lim\limits_{l \to \infty}\frac{\varphi_i(t_{k_l})-\varphi_i(0)-t_{k_l}\varphi_i'(0)}{1/2t_{k_l}^2}= \varphi''_i(0),$$
which yields a contradiction.

It can be derived from the definition of $\varphi$  that
\begin{align*}
  &\lim_{t\downarrow 0}\frac{\varphi(t)}{1/2t^2}=\lim_{t \downarrow0}\frac{g\left(x^*+td+\frac{1}{2}t^2w\right)-g(x^*)-tg'(x^*;d)}{1/2t^2}=g''(x^*;d,w).
\end{align*}
Since vectors $d$ and $w$ are arbitrary,
  then the mapping $g$ is second-order
directionally differentiable at every $x^*\in U$.
Moreover, from (\ref{varphi''=}), we obtain that $ g''(x^*;d,w)=\varphi''_i(0)=g''_i(x^*;d,w)$, i.e.,
  $I(x^*,d,w)=\{i\in I(x^*,d):g''(x^*;d,w)=g''_i(x^*;d,w)\}$.

(ii) We are now ready to demonstrate that the second-order gph-regularity of the $PC^2$ mapping at $x^*$ in direction $d$ by contradiction, i.e.,
 there exists
a path $w:\mathbb{R}_{+}\to \mathbb{R}^n$ satisfying $t w(t)\to 0$ as $t\downarrow 0$ and a sequence $\{t_k\}_{k\in \mathbb{N}}$ with $t_k\downarrow 0$ such that
\begin{align}\label{g(w(tn))}
    &\lim_{k \to \infty}\frac{g\left(x^*+t_k d+\frac{1}{2}t_k^2w(t_k)\right)-g(x^*)-t_k g'(x^*;d)-\frac{1}{2}t_k^2g''(x^*;d,w(t_k))}
{t_k^2} =\alpha\neq 0.
\end{align}
Let $r(t):=\frac{1}{2}t^2w(t)$, it follows from (\ref{g'-gi'}) that $g\left(x^*+td+\frac{1}{2}t^2w(t)\right)\in \{g_i\left(x^*+td+\frac{1}{2}t^2w(t)\right):i\in I(x^*,d)\}$ for all sufficiently small $t>0$.
Since the index set $I(x^*,d)$ is finite, then there exists  a subsequence $\{\tau_l:=t_{k_l}\}_{l\in \mathbb{N}}$ of $\{t_k\}$ and an index $i_0\in I(x^*,d)$ such that
 \begin{align}\label{g=gi0}
     & g\left(x^*+\tau_l d+\frac{1}{2}\tau_l^2w(\tau_l)\right)= g_{i_0}\left(x^*+\tau_l d+\frac{1}{2}\tau_l^2w(\tau_l)\right).
 \end{align}
 It follows from the  second-order gph-regularity of $g_{i_0}$, together with  (\ref{g(w(tn))}) and (\ref{g=gi0}) that
  \begin{align}\label{g''neq gi0''}
 2\alpha =&\lim_{l \to \infty}\frac{g\left(x^*+\tau_l d+\frac{1}{2}\tau_l^2w(\tau_l)\right)-g(x^*)-\tau_l g'(x^*;d)-\frac{1}{2}\tau_l^2g''(x^*;d,w(\tau_l))}
{1/2\tau_l^2}\notag\\
=&\lim_{l \to \infty}\frac{g_{i_0}\left(x^*+\tau_l d+\frac{1}{2}\tau_l^2w(\tau_l)\right)-g_{i_0}(x^*)-\tau_l g_{i_0}'(x^*;d)-\frac{1}{2}\tau_l^2g''(x^*;d,w(\tau_l))}
{1/2\tau_l^2}\notag\\
   =& \lim\limits_{l\to \infty} g_{i_0}''(x^*;d,w(\tau_l))-g''(x^*;d,w(\tau_l))\neq0.
 \end{align}
 Furthermore, based on the above discussion and (\ref{g''neq gi0''}),  there exists  $l_0\in \mathbb{N}$ such that for every $l\ge l_0$, we have
 $$i_0\notin I(x^*,d,w(\tau_l))=\{i\in I(x^*,d):g''(x^*;d,w(\tau_l))=g''_i(x^*;d,w(\tau_l))\}.$$
Similarly as case (i), for any fixed $\tau_l$, $g$ is a continuous selection of the mappings $g_i, i\in I(x^*,d,w(\tau_l))$  and thus there exists an open interval $(0,\beta_l)$ with $\beta_l<\tau_l$ such that for all $u\in (0,\beta_l)$,
%and $g\left(x^*+t d+\frac{1}{2}t^2 w(\tau_l)\right)\neq g_{i_1}\left(x^*+t d+\frac{1}{2}t^2 w(\tau_l)\right)$
 \begin{align}\label{t in(0,beta)}
& g\left(x^*+u d+\frac{1}{2}u^2 w(\tau_l)\right)\in \left\{g_{i}\left(x^*+u d+\frac{1}{2}u^2 w(\tau_l)\right): i\in I(x^*,d,w(\tau_l))\right\}
\end{align}
and   $g\left(x^*+u d+\frac{1}{2}u^2 w(\tau_l)\right)\neq g_{i_0}\left(x^*+u d+\frac{1}{2}u^2 w(\tau_l)\right)$.

For each given $l\ge l_0$, by the continuity of $g\left(x^*+u d+\frac{1}{2}u^2 w(\tau_l)\right)$  with respect to $u$,  and in view of (\ref{g=gi0}) and (\ref{t in(0,beta)}),
  there exists a positive integer $N_l$ and some   indices satisfying
  $i_1\in I(x^*,d,w(\tau_{l})), \{i_2,i_3,\cdots,i_{N_l}\} \subseteq I(x^*)$  and $i_{N_l+1}:=i_0$
   such that
\begin{align*}
    g\left(x^*+u_{l}^{\nu} d+\frac{1}{2}(u_{l}^{\nu})^2 w(\tau_{l})\right)
   &=g_{i}\left(x^*+u_{l}^{\nu} d+\frac{1}{2}(u_{l}^{\nu})^2 w(\tau_{l})\right), i\in\{i_{\nu},i_{\nu+1}\}  ~\text{for}~ \nu\in[N_l],
   \end{align*}
   where the points  $u_{l}^1,u_{l}^2,\ldots,u_{l}^{N_l} \in (0, \tau_{l}]$.
Since the index set $I(x^*)$ is finite, the family of sets $\left\{\{i_{\nu},i_{\nu+1}\}:\nu\in[N_l]\right\}$   is finite as well.

Hence, without loss of generality,
we can find an infinite sequence $\{l_{\kappa}\}_{\kappa\in \mathbb{N}}$, and for each $\kappa$, there exists a finite   positive integer $N_0$ together with
 a common set of indices satisfying
   $i_1\in I(x^*,d,w(\tau_{l_{\kappa}})), \{i_2,i_3,\cdots,i_{N_0}\} \subseteq I(x^*)  $  and $i_{N_0+1}=i_0$
   such that
\begin{align*}
g\left(x^*+u_{l_{\kappa}}^{\nu} d+\frac{1}{2}(u_{l_{\kappa}}^{\nu})^2 w(\tau_{l_{\kappa}})\right)
   &=g_{i}\left(x^*+u_{l_{\kappa}}^{\nu} d+\frac{1}{2}(u_{l_{\kappa}}^{\nu})^2 w(\tau_{l_{\kappa}})\right), i\in\{i_{\nu},i_{\nu+1}\} ~\text{for}~ \nu\in[N_0],
   \end{align*}
with the points  $u_{l_{\kappa}}^1,u_{l_{\kappa}}^2,\ldots,u_{l_{\kappa}}^{N_0} \in (0, \tau_{l_{\kappa}}]$.
Then $u_{l_{\kappa}}^{\nu}\to 0$ and $u_{l_{\kappa}}^{\nu} w(\tau_{l_{\kappa}})\to 0$ as $\kappa\to \infty$ for each $\nu$,
 thus from the definition,   it holds that $g'(x^*;d)=g_{i_1}'(x^*;d)=\cdots=g_{i_{N_0}}'(x^*;d)=g_{i_0}'(x^*;d)$.

Since $g_{i_{\nu}}$ is a $C^2$ mapping and by (\ref{C2 g'= g''=}),      it follows that
\begin{align*}
  &g_{i_{\nu}}\left(x^*+u_{l_{\kappa}}^{\nu} d+\frac{1}{2}(u_{l_{\kappa}}^{\nu})^2 w(\tau_{l_{\kappa}})\right) \\
  = & g_{i_{\nu}}(x^*)+u_{l_{\kappa}}^{\nu}\mathcal{J}g_{i_{\nu}}(x^*)\left(d+\frac{1}{2}u_{l_{\kappa}}^{\nu} w(\tau_{l_{\kappa}})\right)+\frac{1}{2}(u_{l_{\kappa}}^{\nu})^2 d^T\nabla^2g_{i_{\nu}}(x^*)d+o((u_{l_{\kappa}}^{\nu})^2)\\
  =&g_{i_{\nu}}(x^*)+u_{l_{\kappa}}^{\nu} g'_{i_{\nu}}(x^*;d)+\frac{1}{2}(u_{l_{\kappa}}^{\nu})^2 g''_{i_{\nu}}(x^*;d,w(\tau_{l_{\kappa}}))+o((u_{l_{\kappa}}^{\nu})^2).
\end{align*}
 %it follows from the second-order gph-regularity of $g_{i_1}$ and $g^{i_2}$ that
It then holds   that
\begin{align*}
  0= &\lim_{\kappa\to \infty} \frac{g_{i_{\nu}}\left(x^*+u_{l_{\kappa}}^{\nu} d+\frac{1}{2}(u_{l_{\kappa}}^{\nu})^2 w(\tau_{l_{\kappa}})\right)-g_{i_{{\nu}+1}}\left(x^*+u_{l_{\kappa}}^{\nu} d+\frac{1}{2}(u_{l_{\kappa}}^{\nu})^2 w(\tau_{l_{\kappa}})\right)}{1/2(u_{l_{\kappa}}^{\nu})^2}\\
  =&\lim_{\kappa\to\infty}g_{i_{\nu}}''(x^*;d,w(\tau_{l_{\kappa}}))-g''_{i_{{\nu}+1}}(x^*;d,w(\tau_{l_{\kappa}})).
\end{align*}
%By summing the above equalities over $\nu=1,2,\ldots, \nu_0$, with $i_{\nu_0+1}=i_0$,
 Then we obtain that
$$\lim_{\kappa\to\infty}g_{i_1}''(x^*;d,w(\tau_{l_{\kappa}}))-g''_{i_0}(x^*;d,w(\tau_{l_{\kappa}}))=0,$$
 where $i_1\in I(x^*,d,w(\tau_{l_{\kappa}}))$  for each $\kappa$, and it holds that
$$ \lim\limits_{\kappa\to \infty} g''(x^*;d,w(\tau_{l_{\kappa}}))-g_{i_0}''(x^*;d,w(\tau_{l_{\kappa}}))=
    \lim\limits_{\kappa\to \infty} g_{i_1}''(x^*;d,w(\tau_{l_{\kappa}}))-g_{i_0}''(x^*;d,w(\tau_{l_{\kappa}}))=0,
$$
which contradicts with (\ref{g''neq gi0''}).
 This completes the proof.
 \end{proof}

We give the following propositions to further study the properties of second-order gph-regularity.
\begin{proposition}\label{prop tg''(x^*;d,w(t)) to 0}
   Suppose that $g$ is second-order directionally differentiable  at $x^*$ in direction $d$ and locally Lipschitz continuous  with modulus $c$ at
   $x^*$. Assume that there exists a path $w:\mathbb{R}_{+}\to \mathbb{R}^n$ such that $tw(t)\to 0$ as
   $t\downarrow 0$. Then it holds that
$$\lim_{t\downarrow 0} tg''(x^*;d,w(t))= 0.$$
\end{proposition}
\begin{proof}
  It follows from Proposition \ref{prop lipschitz--Hadamard directional differentiable} that for a fixed
  $w^*$ and $t\ge0$, it holds that
  $$ \left\|g''(x^*;d,w(t))-g''(x^*;d,w^*)\right\|\le  c\|w(t)-w^*\|.
  $$
By the above inequality and taking $t\downarrow 0$, we derive
 $$
    \limsup_{t\downarrow 0}\left\|tg''(x^*;d,w(t))-tg''(x^*;d,w^*)\right\|\le \limsup_{t\downarrow 0} c\|t
    w(t)-tw^*\|,
$$
 thus    $tg''(x^*;d,w(t))\to 0$ holds as $t\downarrow 0$. This completes the proof.
\end{proof}

The following proposition demonstrates that the composition of mappings that possess second-order gph-regularity  maintains the same property.

\begin{proposition}\label{prop a composition second-order gph-regular}
Let $g:\mathbb{R}^n\to \mathbb{R}^m$ and $f:\mathbb{R}^m\to
\mathbb{R}^r$.
 Suppose that $g$ and $f$  are second-order gph-regular at $x^*$ and $g(x^*)$, respectively. Then
the composite mapping $h:=f\circ g$ is second-order gph-regular at
$x^*$.
\end{proposition}

\begin{proof}
Let $y^*:=g(x^*)$.  It follows from \cite[Propositions 2.47 and 2.53]{Bonnans} that the composite
mapping $h$ is second-order directionally differentiable at $x^*$ and the corresponding
chain rules
\begin{align}\label{composite chain rule h''}
(f\circ g)'(x^*;d)&=f'(y^*;g'(x^*;d)),\\
  (f\circ g)''(x^*;d,w)&=f''\left(y^*;g'(x^*;d),g''(x^*;d,w)\right)\notag
\end{align}
hold.  Consider some arbitrary direction $d\in \mathbb{R}^n$ and some arbitrary path $w:\mathbb{R}_{+}\to \mathbb{R}^n$ such that
 $tw(t)\to 0$ as $t\downarrow 0$.
 And from Proposition \ref{prop tg''(x^*;d,w(t)) to 0}, we derive that
      $tg''(x^*;d,w(t))\to 0$ as $t\downarrow 0$.
 Since mappings $g$ and $f$ are second-order gph-regular at $x^*$ and $y^*$, respectively,
  it yields
\begin{align*}
&(f\circ g)\left(x^*+td+\frac{1}{2}t^2w(t)\right)\\
=&f\left(g(x^*)+tg'(x^*;d)+\frac{1}{2}t^2g''(x^*;d,w(t))+o(t^2)\right)\\
=& f\left(g(x^*)+tg'(x^*;d)+\frac{1}{2}t^2g''(x^*;d,w(t))\right) + o(t^2)\\
  =&f(y^*)+tf'(y^*;g'(x^*;d))+\frac{1}{2}t^2f''\left(y^*;g'(x^*;d),g''(x^*;d,w(t))\right)+o(t^2)\\
  =&h(x^*)+th'(x^*;d)+\frac{1}{2}t^2h''(x^*;d,w(t))+o(t^2),
\end{align*}
where  the second equality follows from the Lipschitz continuity and
 the last equality holds by (\ref{composite chain rule h''}). It is straightforward to verify that the
$h=f\circ g$ is
second-order gph-regular at $x^*$ for every direction $d$. The proof is completed.
\end{proof}

Let $\{f_i:i\in [n]\}$ be functions  satisfying    second-order gph-regular. From   Example \ref{example min} and   Proposition \ref{prop a composition second-order gph-regular}, it follows that the
 compose function $f(x)=\min\{f_1(x),\cdots,f_n(x)\}$ is  second-order gph-regular.  Many functions can be expressed as the composition of  simple functions and twice continuously differentiable functions, we derived  from the above discussions that the second-order gph-regularity is not difficult to be satisfied.

\begin{proposition}\label{prop f^-1}
  Let $f:\mathbb{R}^n\to \mathbb{R}^n$ be
   second-order gph-regular at $x^*$. If $\partial f(x^*)$ is of maximal rank,  i.e., every element in $\partial f(x^*)$ is of maximal rank, then the locally Lipschitz
  inverse mapping $f^{-1}$ is second-order directionally differentiable and second-order gph-regular at $f(x^*)$.
\end{proposition}
\begin{proof}
    The existence of the locally Lipschitz inverse mapping follows from Clarke inverse function theorem \cite[Theorem 7.1.1]{Clarke}. Moreover, it  follows  from \cite[Lemma 2]{Kummer} that there exist
  neighborhoods $U$ and $V$ of $x^*$ and $f(x^*)$, respectively, and the mapping $f^{-1}:V\to
  \mathbb{R}^n$ is  directionally differentiable at $f(x^*)$, and $d=(f^{-1})'(f(x^*);h)$ if and only if $h=f'(x^*;d)$. Furthermore,  according to \cite{Yin 25}, $f^{-1}$ is second-order directionally differentiable at $f(x^*)$, and $v=f''(x^*;d,w)$ if and only if $w=(f^{-1})''(f(x^*);h,v)$.

Therefore it remains  to verity that $f^{-1}$ is second-order gph-regular at $f(x^*)$.
Consider any direction $h\in \mathbb{R}^n$ and any path $v:\mathbb{R}_{+}\to \mathbb{R}^n$ which satisfies
$tv(t)\to 0$ as
$t\downarrow 0$.  Let
\begin{align}\label{d= ,w(t)=}
   &  d:=(f^{-1})'(f(x^*);h),\quad w(t):=(f^{-1})''(f(x^*);h,v(t)),
\end{align}
then
\begin{align}\label{h=, v(t)=}
  &h=f'(x^*;d),\quad v(t)=f''(x^*;d,w(t)).
\end{align}
And from Proposition \ref{prop tg''(x^*;d,w(t)) to 0}, we derive that
      $$tw(t)=t(f^{-1})''(f(x^*);h,v(t))\to 0,$$ as $t\downarrow 0$.
 Thus the second-order gph-regularity of $f$ at $x^*$ implies
\begin{align*}
  & f\left(x^*+td+\frac{1}{2}t^2w(t)\right)=f(x^*)+tf'(x^*;d)+\frac{1}{2}t^2f''(x^*;d,w(t))+o(t^2).
\end{align*}
     Since $f^{-1}$ is locally Lipschitz continuous at $f(x^*)$, then
\begin{align}%
  &\left\|f^{-1}\left(f(x^*)+tf'(x^*;d)+\frac{1}{2}t^2f''(x^*;d,w(t))\right)
  -f^{-1}\left(f\left(x^*+td+\frac{1}{2}t^2w(t)\right)\right)\right\|\le o(t^2).\notag
\end{align}
 Hence
 \begin{align}\label{f^-1 Lipschitz continuous}
     f^{-1}\left(f(x^*)+tf'(x^*;d)+\frac{1}{2}t^2f''(x^*;d,w(t))\right)
     &=f^{-1}\left(f\left(x^*+td+\frac{1}{2}t^2w(t)\right)\right)+ o(t^2)\notag\\
     &=x^*+td+\frac{1}{2}t^2w(t)+ o(t^2).
     \end{align}
 Plugging (\ref{d= ,w(t)=}) and (\ref{h=, v(t)=}) into (\ref{f^-1 Lipschitz continuous}), it follows that
     \begin{align*}
        &f^{-1}\left(f(x^*)+th+\frac{1}{2}t^2v(t)\right)\\
        =&f^{-1}\left(f(x^*)\right)+t(f^{-1})'(f(x^*);h)+\frac{1}{2}t^2(f^{-1})''(f(x^*);h,v(t))+
        o(t^2).
     \end{align*}
     Therefore $f^{-1}$ is second-order gph-regular at $f(x^*)$ for every direction $h$. The proof is completed.
\end{proof}

Let $H:\mathbb{R}^n \times \mathbb{R}^m\to  \mathbb{R}^m$ be a given mapping satisfying second-order directionally differentiable and second-order gph-regular.
  The projection of the Clarke generalized Jacobian onto the $y$-component is denoted by $\pi_y\partial H(x,y)$, as the set
$$\{Z_2\in \mathbb{R}^{m\times m}: \text{there exists some}~ Z_1\in \mathbb{R}^{m\times n} ~\text{such that}~  [Z_1~Z_2]\in \partial H(x,y)\}.$$
Denote by $y(x)$ the implicit solution mapping  of the system $H(x,y)=0$  under Clarke implicit function theorem.
In the rest of this section, we demonstrate that  $y(x)$
 is  second-order directionally differentiable and  second-order gph-regular.

\begin{corollary}\label{cor implicit function outer second-order gph-regular}
Let $(x^*, y^*) \in \mathbb{R}^n \times \mathbb{R}^m$ be a point at which the mapping $H$ is   second-order directionally differentiable and second-order gph-regular.
  If $\pi_y\partial H(x^*,y^*)$  is of maximal
rank and $H(x^*,y^*)=0$,
 then there exist constants
$\delta>0$ and $\varepsilon>0$, and  a locally Lipschitz continuous mapping $
y(\cdot):\mathbb{B}_{\delta}(x^*)\to\mathbb{B}_{\varepsilon}(y^*)$ such that $y(x^*)=y^*$ and $$H(x,y(x))=0, \quad \forall x\in\mathbb{B}_{\delta}(x^*).$$
 Moreover, the mapping $y(x)$ is second-order directionally differentiable and second-order gph-regular at
$x^*$.
\end{corollary}

\begin{proof}
The first part of the corollary follows from Clarke   implicit function theorem \cite[p255]{Clarke}.
Therefore, it remains to prove that $y(x)$ is second-order directionally differentiable and second-order gph-regular at $x^*$.

      Define $\phi$ from $\mathbb{R}^{n+m}$ to itself by
            $$\phi(x,y):=[x;H(x,y)],$$
            where $[x;H(x,y)]:=[x^T,H(x,y)^T]^T$.
             By
         Proposition \ref{prop a composition second-order gph-regular}, it follows that  $\phi$ is
           second-order directionally differentiable and second-order gph-regular at every point.
By the chain rule, we derive that
 $$\partial
\phi(x^*,y^*)\subseteq\left\{\left[\begin{array}{cc}
                                                     I_n &0 \\
                                                    Z_1 & Z_2
                                                   \end{array}\right]:[Z_1~Z_2]\in \partial
                                                   H(x^*,y^*)\right\}.$$
  Since $\pi_y\partial H(x^*,y^*)$  is of maximal
rank, then $\partial \phi(x^*,y^*)$  is  also of maximal
rank.
             By
         Proposition \ref{prop f^-1}, we derive that $\phi^{-1}$ is also second-order directionally
         differentiable and second-order gph-regular at $\phi(x^*,y^*)=[x^*;0]$.
      Notice that
           $$\phi^{-1}\circ \phi(x,y)=\phi^{-1}([x;H(x,y)])=(x,y)$$
 holds locally around $(x^*,y^*)$. By plugging $y=y(x)$ into the above equality, we derive that
 $$\phi^{-1}([x;0])=(x,y(x)),$$
 then $\phi^{-1}([x;0])$ is  second-order gph-regular at $x^*$.
  Indeed, it only takes the direction $d=(d_x,0)$ and the path $w=(w_x,0)$ in Definition \ref{def Second-order gph-regular} to get.

Therefore the solution mapping $y(x)$, which is  the component of $\phi^{-1}([x;0])$,
         is second-order
directionally differentiable and second-order gph-regular at $x^*$. The proof is completed.
         \end{proof}

\section{Optimality conditions for nonsmooth optimization problems by directional derivatives}\label{section 3}

In this section, we analyze the first- and second-order necessary optimality conditions and  second-order
 sufficient optimality conditions for the nonsmooth optimization problem (P).

\begin{definition}\label{def MSCQ}(MSCQ). Let $x^*\in \Psi$. We say that
the MSCQ for $\Psi$ holds at $x^*$ if there exist a neighborhood
$U$ of $x^*$ and a constant $\kappa>0$ such that
\begin{align}\label{MSCQ}
  d(x,\Psi)\le \kappa d(G(x),K),\ \forall x\in U.
\end{align}
\end{definition}

\begin{proposition}\label{prop-tangent cone=linearized}
 Let $K$ be a closed subset of $\mathbb{R}^m$ and a mapping $G:\mathbb{R}^n\to \mathbb{R}^m$ be locally Lipschitz
 continuous and directionally differentiable at $x^*$. Suppose that MSCQ for $\Psi$ holds at $x^*$, then
  \begin{align}\label{tangent cone}
   &\mathcal{T}_{\Psi}(x^*)=\{d\in \mathbb{R}^n:G'(x^*;d)\in \mathcal{T}_{K}(G(x^*))\}.
  \end{align}
  \end{proposition}
\begin{proof}
  For any point $d\in \mathcal{T}_{\Psi}(x^*)$, there exist sequences $t_k\downarrow 0$ and $d^k\to d$ such that
   $x^*+t_k d^k\in \Psi$, i.e., $G(x^*+t_k d^k)\in K$. Therefore, and since the mapping $G$ is locally Lipschitz
   continuous and directionally differentiable at $x^*$, by \cite[Proposition 2.49]{Bonnans}, it follows that
   $$G(x^*+t_k d^k)=G(x^*)+t_k G'(x^*;d)+o(t_k)\in K,$$ i.e., $G'(x^*;d)\in \mathcal{T}_{K}(G(x^*))$.

   Conversely, for any point $d$ satisfying $G'(x^*;d)\in \mathcal{T}_{K}(G(x^*))$, there exists a sequence
   $t_k\downarrow 0$ such that $d(G(x^*)+t_k G'(x^*;d),K)=o(t_k)$.  It follows  that
   $$
   d(G(x^*+t_k d),K)\le d(G(x^*)+t_k G'(x^*;d),K)+o(t_k)=o(t_k).
   $$
   Thus it follows from (\ref{MSCQ}) that $d(x^*+t_k d,\Psi)= o(t_k)$, i.e., $d\in \mathcal{T}_{\Psi}(x^*)$.
 \end{proof}

\begin{theorem}[First-order necessary optimality conditions]\label{thm FONC for nonsmooth}
 Let $x^*$ be a locally optimal solution of $(P)$. Suppose that $K$ is a closed subset of $\mathbb{R}^m$ and
 $f$ and $G$ are locally Lipschitz continuous and directionally differentiable at $x^*$. Then the
 following assertions hold:
 \begin{itemize}
   \item[(i)]The point $d=0$ is an optimal solution of the problem
   \begin{align}\label{p 1}
      & \min~  f'(x^*;d)~~~{\rm s.t.}~ d\in \mathcal{T}_{\Psi}(x^*).
   \end{align}
   \item [(ii)]If MSCQ for $\Psi$ holds at $x^*$, then $d=0$ is an optimal solution of the problem
   \begin{align}%\label{p 2}
      & \min ~ f'(x^*;d)~~~{\rm s.t.}~G'(x^*;d)\in \mathcal{T}_{K}(G(x^*)).\notag
   \end{align}
 \end{itemize}
\end{theorem}
\begin{proof}
  (i) Let $d\in \mathcal{T}_{\Psi}(x^*)$. Then by the definition of the tangent set, there
exist sequences $t_k\downarrow 0$ and $d^k\to d$ such that $x^*+t_kd^k\in \Psi$. Since $x^*$ is a
locally optimal solution of $(P)$ and $f$ is locally Lipschitz continuous and directionally differentiable at
$x^*$, it follows that
\begin{align*}
   & 0\le \lim_{k\to \infty}\frac{f(x^*+t_k d^k)-f(x^*)}{t_k}=f'(x^*;d),
\end{align*}
and $f'(x^*;0)=0$, consequently $d=0$ is an optimal solution of the problem (\ref{p 1}).

(ii) By Proposition \ref{prop-tangent cone=linearized}, we have that if MSCQ for $\Psi$ holds at $x^*$, then
(\ref{tangent cone}) holds.
Combining this with assertion (i), we obtain the conclusion.
\end{proof}

 \begin{proposition}\label{prop-second-order tangent cone=linearized}
   Let $K$ be a closed subset of $\mathbb{R}^m$ and
 $G:\mathbb{R}^n\to \mathbb{R}^m$ be locally Lipschitz continuous and second-order directionally
 differentiable at $x^*$. Suppose that MSCQ for $\Psi$ holds at $x^*$, then for any $d\in \mathbb{R}^n$,
  \begin{align}\label{second-order tangent set=}
   &\mathcal{T}^2_{\Psi}(x^*,d)=\left\{w\in \mathbb{R}^n:G''(x^*;d,w)\in
   \mathcal{T}_{K}^2(G(x^*),G'(x^*;d))\right\}.
  \end{align}
\end{proposition}
\begin{proof}
 For any $d\in \mathcal{T}_{\Psi}(x^*)$, consider any point $w\in \mathcal{T}^2_{\Psi}(x^*,d)$, there exists a sequence $t_k\downarrow 0$ and let
  $x^k:=x^*+t_k d+\frac{1}{2}t_k^2w$ be the corresponding parabolic sequence, such that
  $d(x^k,\Psi)=o(t_k^2)$. From Proposition \ref{prop-tangent
  cone=linearized}, we have that
  $G'(x^*;d)\in \mathcal{T}_{K}(G(x^*))$. Since $G$ is second-order directionally differentiable at $x^*$, we have
  \begin{align}\label{G second-order parabolic expansion}
    G(x^k)=G(x^*)+t_k G'(x^*;d)+\frac{1}{2}t_k^2G''(x^*;d,w)+o(t_k^2).
  \end{align}
  Since $G$ is locally Lipschitz continuous at $x^*$
 (with modulus $L_G>0$), it follows that
  $$d(G(x^k),K)=\inf_{y\in\Psi}\|G(x^k)-G(y)\|\le L_G\cdot\inf_{y\in\Psi}\|x^k-y\|=L_G\cdot d(x^k,\Psi)=o(t_k^2),$$
  thus $d(G(x^k),K)\le o(t_k^2)$.
   Together with (\ref{G second-order parabolic expansion}) this implies that
  \begin{align}\label{G-second-order tangent set}
   G''(x^*;d,w)\in \mathcal{T}_{K}^2(G(x^*),G'(x^*;d)).
  \end{align}

The inverse inclusion can be obtained by applying the above arguments in the inverse order and using the MSCQ, similarly as the proof of  Proposition \ref{prop-tangent cone=linearized}.
\end{proof}

The chain rule (\ref{second-order tangent set=}) was derived in \cite[Lemma 2.5]{Mehlitz 20} under the MSCQ for the case that the mapping $G$ is  twice continuously differentiable.

\begin{theorem}[Second-order necessary optimality conditions]\label{thm SONC for nonsmooth}
 Let $x^*$ be a locally optimal solution of $(P)$. Suppose that $K$ is a closed subset of  $\mathbb{R}^m$ and
 $f$ and $G$ are locally Lipschitz continuous and second-order directionally differentiable at
 $x^*$. Then the following assertions hold:
 \begin{itemize}
   \item[(i)]For every $d\in \mathcal{T}_{\Psi}(x^*)$ with $f'(x^*;d)\le 0$ and all $w\in \mathcal{T}_{\Psi}^2(x^*,d)$, it follows that
   \begin{align*}
      f''(x^*;d,w)\ge 0.
   \end{align*}
   \item [(ii)]If MSCQ for $\Psi$ holds at $x^*$, then
    for every $d\in \mathcal{C}(x^*)$ and all $w\in
       \mathbb{R}^n$ satisfying (\ref{G-second-order tangent set}), it follows that
   \begin{align*}
      &  f''(x^*;d,w)\ge 0,
   \end{align*}
   where   $\mathcal{C}(x^*):=\left\{d\in \mathbb{R}^n: G'(x^*;d)\in \mathcal{T}_{K}(G(x^*)),f'(x^*;d)\le
   0\right\}$ denotes the critical cone of
   $(P)$ at $x^*$.
 \end{itemize}
\end{theorem}
\begin{proof}
  (i) Consider $d\in \mathcal{T}_{\Psi}(x^*)$ with $f'(x^*;d)\le 0$ and $w\in \mathcal{T}_{\Psi}^2(x^*,d)$,
    there exist sequences $t_k\downarrow 0$ and $w^k\to w$ such that $x^*+t_k d+\frac{1}{2}t_k^2w^k\in \Psi$.
From  Theorem \ref{thm
  FONC for nonsmooth}, we can obtain $f'(x^*;d)=0$.
   Since $f$ is locally Lipschitz continuous and second-order
  directionally differentiable at $x^*$, we have
  \begin{align}
 & 0\le \lim_{k\to \infty}\frac{f(x^*+t_k
 d+\frac{1}{2}t_k^2w^k)-f(x^*)-t_kf'(x^*;d)}{\frac{1}{2}t_k^2}=f''(x^*;d,w)\notag
  \end{align}

  (ii) By Propositions \ref{prop-tangent cone=linearized} and \ref{prop-second-order tangent cone=linearized}, we
  have that if MSCQ for $\Psi$ holds at $x^*$, then (\ref{tangent cone}) and (\ref{second-order tangent set=})
  hold. Combining these with assertion (i), we obtain the conclusion.
\end{proof}

Based on the second-order  gph-regularity of the constraint mapping, we show the outer second-order regularity of the feasible region of the nonsmooth problem (P) under the MSCQ, which is an extension of \cite[Proposition 3.88]{Bonnans}.

\begin{proposition}\label{prop G^-1(K) second order regular}
  Let $K$ be a closed subset of $\mathbb{R}^m$ and $G:\mathbb{R}^n\to \mathbb{R}^m$ be  
   second-order gph-regular at $x^*$. Suppose
  that MSCQ for $\Psi$ holds at $x^*$. If $K$ is outer second-order regular at $G(x^*)$ in direction
  $G'(x^*;d)$, then the set $\Psi$ is outer second-order regular at $x^*$ in direction $d$. Furthermore,
  if $K$ is outer second-order regular at $G(x^*)$, then the set $\Psi$ is outer second-order regular at $x^*$.
\end{proposition}
\begin{proof}
  Let $x^k:=x^*+t_k d+\frac{1}{2}t_k^2w^k\in \Psi$ be a sequence such that $t_k\downarrow 0$ and $t_k w^k\to 0$.
 If there exists a constant $\kappa>0$ such that for all sufficiently large $k$,   the  following condition holds:
\begin{align}\label{d(w^k,T^2)}
  d\left(w^k, \mathcal{T}^2_{\Psi}(x^*,d)\right)\le \kappa d\left(G''(x^*;d,w^k),
  \mathcal{T}_{K}^2\left(G(x^*),G'(x^*;d)\right)\right),
\end{align}
 then the first part of the proposition holds automatically.
Indeed,  since $G$ is second-order gph-regular at  $x^*$, then
\begin{align}%\label{G(x^k)}
   & G(x^k)=G(x^*)+t_k G'(x^*;d)+\frac{1}{2}t_k^2G''(x^*;d,w^k)+o(t_k^2)\in K.\notag
\end{align}
Therefore the outer second-order regularity of $K$ implies
\begin{align*}
   & d\left(G''(x^*;d,w^k), \mathcal{T}_{K}^2\left(G(x^*),G'(x^*;d)\right)\right)\to 0
\end{align*}
and thus $ d\left(w^k, \mathcal{T}^2_{\Psi}(x^*,d)\right)\to 0$. Consequently, $\Psi$ is outer
second-order regular at $x^*$ in direction $d$.

We only need to verify (\ref{d(w^k,T^2)}).
  Assume there exists $v\in \mathcal{T}_{K}^2\left(G(x^*),G'(x^*;d)\right)$ such that
 $\|G''(x^*;d,w^k)-v\|=d\left(G''(x^*;d,w^k),
  \mathcal{T}_{K}^2\left(G(x^*),G'(x^*;d)\right)\right)$ and for
some sequences $s_l\downarrow 0$ and $ v^l\to v$, it follows that
$G(x^*)+s_l G'(x^*;d)+\frac{1}{2}s_l^2v^l\in K$.
  Then from MSCQ and the second-order  directional differentiability of $G$, we have that
\begin{align}\label{d(G''(x^*;d,w^k),T^2)}
   &\frac{1}{\kappa}\liminf_{l\to \infty}\frac{2}{s_l^2}d\left(x^*+s_ld+\frac{1}{2}s_l^2w^k,\Psi\right)\notag\\
   \le~~ &\liminf_{l\to \infty}\frac{2}{s_l^2}d\left(G(x^*+s_ld+\frac{1}{2}s_l^2w^k),K\right)\notag\\
    \le~~ &\liminf_{l\to \infty}\left[\frac{2}{s_l^2}d\left(G(x^*)+s_lG'(x^*;d)+\frac{1}{2}s_l^2G''(x^*;d,w^k),K\right)+o(1) \right]\notag\\
   \le~~ &\liminf_{l\to \infty} \frac{2}{s_l^2}\left\|G(x^*)+s_lG'(x^*;d)+\frac{1}{2}s_l^2G''(x^*;d,w^k)
-\left[G(x^*)+s_lG'(x^*;d)+\frac{1}{2}s_l^2v^l\right]\right\| \notag\\
  = ~~ &\liminf_{l\to \infty}\|G''(x^*;d,w^k)-v^l\|\notag\\
  =~~ & \|G''(x^*;d,w^k)-v\|\notag\\
		=~~ & d\left(G''(x^*;d,w^k),
		\mathcal{T}_{K}^2\left(G(x^*),G'(x^*;d)\right)\right).
\end{align}

By the definition of outer second-order tangent set, we have
 $$\limsup_{l\to \infty}\frac{\Psi-x^*-s_ld}{\frac{1}{2}s_l^2}\subseteq\mathcal{T}^2_{\Psi}(x^*,d),$$
 thus by \cite[Corollary 4.7]{Rockafellar}, it follows that
 \begin{align}%\label{G^-1(K) second order regular}
  d\left(w^k,\mathcal{T}^2_{\Psi}(x^*,d)\right)&\le \liminf_{l\to
  \infty}d\left(w^k,\frac{\Psi-x^*-s_ld}{\frac{1}{2}s_l^2}\right)\notag\\
  &= \liminf_{l\to \infty}\frac{2}{s_l^2}d\left(x^*+s_ld+\frac{1}{2}s_l^2w^k,\Psi\right)\notag\\
  &\le \kappa d\left(G''(x^*;d,w^k), \mathcal{T}_{K}^2\left(G(x^*),G'(x^*;d)\right)\right),\notag
 \end{align}
 where the last inequality holds by (\ref{d(G''(x^*;d,w^k),T^2)}).

  Furthermore,
  if K is outer second-order regular at $G(x^*)$ for any direction from $\mathbb{R}^m$, then the set $\Psi$ is
  outer second-order regular at $x^*$ for any direction from $\mathbb{R}^n$.
 We complete the proof.
\end{proof}

 In the rest of this section, we construct the second-order sufficient optimality conditions for the nonsmooth problem.

\begin{theorem}[Second-order sufficient optimality conditions]\label{thm SOSC for
nonsmooth}
Let $x^*$ be a feasible point of $(P)$. Suppose that $f$ is  %locally Lipschitz continuous, second-order directionally differentiable and
    second-order
 epi-regular at $x^*$ and the set $\Psi$ is outer second-order regular at $x^*$ in every direction  $d\in T_{\Psi}(x^*)$.
 Assume that  we have
 \begin{align}\label{thm f' ge 0}
   & f'(x^*;d)\ge 0, \quad \forall d\in T_{\Psi}(x^*),
 \end{align}
 and  the optimization problem
   \begin{align}\label{thm p f''}
    \min_{w}~ f''(x^*;d,w) ~~~{\rm s.t.}~w\in \mathcal{T}_{\Psi}^2(x^*,d),
   \end{align}
possesses a positive optimal objective value for every $d\in C_{\Psi}(x^*):=\left\{d\in T_{\Psi}(x^*): f'(x^*;d)\le
   0\right\}  \backslash \{0\}$. Then the
second-order growth condition holds for $(P)$ at $x^*$.
\end{theorem}
\begin{proof}
 Assume to the contrary that the second-order growth condition does not hold for $(P)$ at $x^*$.
  Then there exists a sequence $x^k\in \Psi$, $x^k\neq x^*$, converging to $x^*$ such that
\begin{align}\label{f contradiction}
  & f(x^k)-f(x^*)\le o(\|x^k-x^*\|^2).
\end{align}
  Select a subsequence if necessary, we define $t_k=\|x^k-x^*\|$, $t_k\downarrow 0$ and $t_k>0$. Let
$d^k=\frac{x^k-x^*}{t_k}$, then $x^k=x^*+t_k d^k$ and $\|d^k\|=1$. Without loss of generality, we have that $d^k\to d\in T_{\Psi}(x^*) $
and $\|d\|=1$.

Since $f$ is locally Lipschitz continuous and directionally differentiable at $x^*$,
then
\begin{align*}
   & f'(x^*;d)=\lim_{k\to \infty}\frac{f(x^*+t_k d^k)-f(x^*)}{t_k}
   \le \lim_{k\to \infty}\frac{o(\|x^k-x^*\|^2)}{t_k}=0,
\end{align*}
thus $d\in \mathcal{C}_{\Psi}(x^*)$ and by (\ref{thm f' ge 0}), we have $f'(x^*;d)=0$.
Let $w^k=\frac{x^k-x^*-t_k d}{\frac{1}{2}t_k^2}$, then $x^k=x^*+t_k d+\frac{1}{2}t_k^2w^k\in \Psi$ and $t_kw^k\to
0$. Since $\Psi$ is outer second-order regular
at $x^*$ in  direction $d\in T_{\Psi}(x^*)$, we have
 \begin{align}\label{d(w^k,T^2) to 0}
   &d(w^k,\mathcal{T}^2_{\Psi}(x^*,d))\to 0.
 \end{align}
  Since $f$ is second-order  epi-regular at $x^*$, then
 $$f(x^*+t_k d+\frac{1}{2}t_k^2w^k) \ge f(x^*)+t_k f'(x^*;d)+\frac{1}{2}t_k^2 f''(x^*;d,w^k)+o(t_k^2).$$
 By
 (\ref{f contradiction}) and the fact $f'(x^*;d)=0$, we have
 \begin{align}\label{limsup f''(x^*;d,w^k)}
   & \limsup_{k\to \infty}f''(x^*;d,w^k)\le 0.
 \end{align}
 Notice that the optimal value of Problem (\ref{thm p f''}) is positive,
 there exists some constant $\beta>0$ such that
 \begin{align*}
    & \forall w\in \mathcal{T}^2_{\Psi}(x^*,d),~~~f''(x^*;d,w)\ge \beta >0.
 \end{align*}
It follows from Proposition \ref{prop lipschitz--Hadamard directional differentiable} that $f''(x^*;d,\cdot)$ is
Lipschitz continuous, then there exists some constant $L_f>0$ such that
\begin{align}\label{Lf}
   & |f''(x^*;d,w^k)-f''(x^*;d,w)|\le L_{f}\|w^k-w\|.
\end{align}
By (\ref{d(w^k,T^2) to 0}), there exist $v^k\in \mathcal{T}^2_{\Psi}(x^*,d)$ such that $\|w^k-v^k\|\le
\frac{\beta}{2L_f}$ for all $k$ large enough.
Formula (\ref{Lf}) yields
\begin{align*}
   & f''(x^*;d,w^k)\ge f''(x^*;d,v^k)-L_f\|w^k-v^k\|\ge \beta-\frac{\beta}{2}=\frac{\beta}{2}>0,
\end{align*}
%thus we obtain
%\begin{align*}
%   & \liminf_{n\to \infty}f''(x^*;d,w^k)\ge \frac{\beta}{2}>0.
%\end{align*}
 which contradicts with (\ref{limsup f''(x^*;d,w^k)}). We complete the proof.
\end{proof}
\begin{corollary}\label{cor SOSC for
nonsmooth}
Let $x^*$ be a feasible point of $(P)$ and $K$ be a closed subset of $\mathbb{R}^m$. Suppose that $f$ is
    second-order
 epi-regular  and $G$ is
   second-order gph-regular at $x^*$.
   Suppose
  that MSCQ for $\Psi$ holds at $x^*$ and $K$ is outer second-order regular at $G(x^*)$ in every direction
  $G'(x^*;d)\in \mathcal{T}_{K}(G(x^*))$.
 Assume that for every direction $d$ satisfying $G'(x^*;d)\in \mathcal{T}_{K}(G(x^*))$, we have
 \begin{align}%\label{f' ge 0}
   & f'(x^*;d)\ge 0,\notag
 \end{align} and  the optimization problem
   \begin{align}%\label{p f''}
    \min_{w}~ f''(x^*;d,w) ~~~{\rm s.t.}~G''(x^*;d,w)\in \mathcal{T}_{K}^2(G(x^*),G'(x^*;d)),\notag
   \end{align}
possesses a positive optimal objective value for every $d\in \mathcal{C}(x^*)\backslash \{0\}$. Then the
second-order growth condition holds for $(P)$ at $x^*$.
 \end{corollary}
\begin{proof}
It follows from Propositions \ref{prop-tangent cone=linearized} and \ref{prop-second-order tangent cone=linearized} that
\begin{align*}
   &\mathcal{T}_{\Psi}(x^*)=\{d\in \mathbb{R}^n:G'(x^*;d)\in \mathcal{T}_{K}(G(x^*))\},\\
  &\mathcal{T}^2_{\Psi}(x^*,d)=\left\{w\in \mathbb{R}^n:G''(x^*;d,w)\in
   \mathcal{T}_{K}^2(G(x^*),G'(x^*;d))\right\}.
  \end{align*}
Moreover, it follows from Proposition \ref{prop G^-1(K) second order regular} that $\Psi$ is outer second-order regular
at $x^*$ in every direction $d\in T_{\Psi}(x^*)$. Hence the proof is completed by Theorem \ref{thm SOSC for nonsmooth}.
\end{proof}

Notice that in \cite[Theorem 4.1]{Ruckmann}, the second-order sufficient optimality condition only holds for
inequality constrained optimization problem. The above theorem extends to the more general case by the concept of second-order gph-regularity.

\section{Optimality conditions for the bi-level problem}\label{section 4}
In this section, we analyze the first-  and  second-order optimality conditions for the bi-local solution of the bilevel programming problem:
\begin{eqnarray*}
(BP)~~~~\min && F(x,y)\\
{\rm s.t.}&& H(x,y)= 0,\ G(x,y)\le 0,\\
&& y\in S(x),
\end{eqnarray*}
where $ S(x)$ denotes the solution set of the lower-level problem:
\begin{eqnarray*}
(P_x)~~~~\min_y && f(x,y)\\
{\rm s.t.} &&
h(x,y)= 0,\  g(x,y)\le 0,
\end{eqnarray*}
where  $F:\mathbb{R}^n \times \mathbb{R}^m \to \mathbb{R} $, $G :\mathbb{R}^n\times \mathbb{R}^m\to \mathbb{R}^q $, $H :\mathbb{R}^n\times \mathbb{R}^m\to \mathbb{R}^p $,
$f:\mathbb{R}^n \times \mathbb{R}^m \to \mathbb{R} $,  $g :\mathbb{R}^n\times \mathbb{R}^m\to \mathbb{R}^s $, $h :\mathbb{R}^n\times \mathbb{R}^m\to \mathbb{R}^r $.
Assume that $f, g, h$ are
 thrice continuously differentiable  and $F, G, H$ are twice continuously
differentiable around $(x^*,y^*)$.
Note that $(\mu^*,\xi^*)\in \mathbb{R}^r\times\mathbb{R}^s$ is the corresponding Lagrangian multiplier such that
the KKT conditions hold at $y^*$ for  $(P_{x^*})$, i.e.,
 \begin{eqnarray}\label{kkt}
 && \nabla_y \mathcal{L}\left(x^*; y^*, \mu^*,\xi^*\right)=0,\nonumber \\
&& h\left(x^*, y^*\right)=0,\\
&& 0\le \xi^*\bot g(x^*,y^*)\leq 0,\nonumber
\end{eqnarray}
where  $\mathcal{L}(x;y,\mu,\xi):=f(x,y)+\mu^Th(x,y)+\xi^Tg(x,y)$ is the Lagrangian function of $(P_x)$. We use $\Lambda_{x^*}(y^*)$ to denote the
set of all $(\mu^*,\xi^*)$ satisfying (\ref{kkt}) and
define $Y(x):=\{y\in \mathbb{R}^m: h(x,y)=0, g(x,y)\le0\}$ as the feasible set of $(P_x)$.
We give the following assumptions:
\begin{itemize}
   %$f, g, h$ are
% twice continuously differentiable  and $F, G, H$ are continuously
%differentiable around $(x^*,y^*)$.
 \item[(A1)] The MFCQ holds at $y^*$ for $(P_{x^*})$; namely, the set $\{\nabla_y h_i(x^*, y^*): i\in [r]\}$ is linearly independent and there exists $d_y\in \mathbb{R}^m$ such that $\mathcal{J}_y h(x^*,y^*)d_y=0$ and $\nabla_y g_i(x^*,y^*)d_y< 0$ for all
      $i\in I_{x^*}(y^*):=\{i\in[s]:g_i(x^*,y^*)=0\}.$
  \item[(A2)] The SSOSC holds at every $(y^*, \mu^*,\xi^*)\in \Lambda_{x^*}(y^*)$ for $(P_{x^*})$; namely,
$$
\left\langle\nabla_{y y}^2 \mathcal{L}\left(x^*; y^*, \mu^*,\xi^*\right) d_y, d_y\right\rangle>0, \quad \forall
d_y \in  \text{aff}~ \mathcal{C}_{x^*}(y^*)\backslash\{0\}, $$
where  $\mathcal{C}_{x^*}(y^*)$ is the critical cone of Problem $(P_{x^*})$ at $y^*$,
\begin{align*}
  \mathcal{C}_{x^*}(y^*)=\{d_y :  \mathcal{J}_y h(x^*, y^*)d_y=0;\nabla_y g_i(x^*,y^*)^Td_y\le 0,i\in I_{x^*}(y^*);
 \nabla_y f(x^*,y^*)^Td_y\le 0 \}.
\end{align*}
\item[(A3)] The CRCQ holds at $y^*$ for $(P_{x^*})$; namely, there exists a neighborhood $U$ of $(x^*,y^*)$ such that for any subsets $I$ of $I_{x^*}(y^*)$ and $J$ of $[r]$, the family of gradient vectors $\{\nabla_y h_i(x, y): i\in J\}\cup\{\nabla_y g_i(x, y):i\in I\}$ has the same rank (depending on $I$ and $J$) for all vectors $(x,y)\in U$.
  \item[(A4)] The LICQ %linear independence constraint qualification
 holds at $y^*$ for $(P_{x^*})$; namely, the set of vectors
$
\{\nabla_y h_i(x^*, y^*): i\in [r]\}\cup\{\nabla_y g_i(x^*, y^*):i\in
I_{x^*}(y^*)\}
$
is linearly independent.
\end{itemize}

When $(\mu^*,\xi^*)\in \Lambda_{x^*}(y^*)$ is a Lagrangian multiplier of Problem $(P_{x^*})$, then
 the affine space of the critical cone $\mathcal{C}_{x^*}(y^*)$ can be expressed as
\begin{align*}
 \text{aff}~ \mathcal{C}_{x^*}(y^*)=\{d_y\in \mathbb{R}^m: &\mathcal{J}_y h(x^*, y^*)d_y=0;\nabla_y
 g_i(x^*,y^*)^Td_y= 0,i\in I_{x^*}^{+}(y^*) \},
\end{align*}
where $I_{x^*}^{+}(y^*):=\{i\in[s]:g_i(x^*,y^*)=0, \xi_{i}^*>0\}$.

\begin{lemma}\label{lem MFCQ SSOSC}
  Suppose that assumptions \textnormal{(A1)} and \textnormal{(A2)} hold. Then there exists constants
$\delta_1>0$ and $\varepsilon_1>0$, and a  continuous mapping $y(\cdot):\mathbb{B}_{\delta_1}(x^*)\to\mathbb{B}_{\varepsilon_1}(y^*)$ such that $y(x^*)=y^*$ and $y(x)$ is the unique local solution of the lower-level problem $(P_x)$ and is  directionally differentiable for $x\in \mathbb{B}_{\delta_1}(x^*)$. Moreover,  $(x^*,y^*)$ is a bi-local solution of (BP) if and only if  $x^*$ is a local solution of the implicit function reformulation  (SP) as follows:
\begin{equation*}
(SP)~~~~\min_x~~~~ F(x,y(x))~~~~{\rm s.t.}~~~~ H(x,y(x))= 0,~ G(x,y(x))\le 0,
\end{equation*}
which substitutes  the unique local solution derived from the lower-level problem into the corresponding
upper-level problem.
\end{lemma}
\begin{proof}
 The first part of the lemma follows from \cite[Theorem 1]{Ralph and Dempe 95}.
According to \cite[Proposition 5.37]{Bonnans}, under the MFCQ and SSOSC  for the lower-level problem $(P_{x^*})$, the uniform second-order quadratic growth condition holds at $y^*$.
The remainder of the proof proceeds similarly to the argument in \cite[Theorem 3.5]{Liu 25}.
\end{proof}

By replacing MFCQ with LICQ in above Lemma and applying \cite[Remark 3.6]{Liu 25}, we obtain that
$(x^*,y^*)$  serves as  a bi-local solution for (BP), which is equivalent to
  $(x^*,y^*,\mu^*,\xi^*)$ being a local solution of the first-order reformulation (FP) as follows:
 \begin{eqnarray*}
(FP)~~~~~~\min_{x,y,\mu,\xi} && F(x,y)\\
{\rm s.t.}&& H(x,y)= 0,G(x,y)\le 0,\\
&&\nabla_{y} \mathcal{L}(x;y,\mu,\xi)=0,\\
&&h(x,y)=0,\\
&&g(x,y )-\Pi_{\mathbb{R}^s_{-}}(g(x,y)+\xi)=0,
\end{eqnarray*}
which is by substituting the lower-level solution set constraint for the KKT conditions
of the lower-level problem.

\begin{lemma}\label{lem MFCQ SSOSC CRCQ}
  Suppose that assumptions \textnormal{(A1)}-\textnormal{(A3)} hold. Then there exist positive constants
$\delta_2<\delta_1$ and $\varepsilon_2<\varepsilon_1$, and a locally Lipschitz continuous mapping $y(\cdot):\mathbb{B}_{\delta_2}(x^*) \to \mathbb{B}_{\varepsilon_2}(y^*)$ such that $y(x^*)=y^*$ and $y(x)$ is
the unique local solution of the lower-level problem $(P_x)$ and   is second-order directionally differentiable and  second-order gph-regular for $x \in \mathbb{B}_{\delta_2}(x^*)$.
  Moreover,  $(x^*,y^*)$ is a bi-local solution of (BP) if and only if  $x^*$ is a local solution of (SP).
\end{lemma}

\begin{proof}
  It follows from \cite[Theorem 2]{Ralph and Dempe 95} that $y(x)$ is locally Lipschitz continuous and $PC^1$ for $x \in \mathbb{B}_{\delta_2}(x^*)$.
  %and the continuous differentiability  is established via the implicit function theorem in  its proof. Under %the assumption \textnormal{(A2)}, this further implies that $y(x)$ is a PC$^2$ mapping.
 Moreover it follows from the proof of \cite[Proposition 7]{Ralph and Dempe 95} and the second-order implicit function theorem \cite[page 364]{Lang 93}  that $y(x)$ is a PC$^2$ mapping.
   Then by Proposition \ref{prop PC2},
  $y(x)$ is second-order directionally differentiable and  second-order gph-regular for all $x \in \mathbb{B}_{\delta_1}(x^*)$. The remaining conclusion holds from Lemma \ref{lem MFCQ SSOSC}.
\end{proof}

For ease of notation, we  define $\mathcal{A}(x,W)$ and $\mathcal{H}(x,W)$ for a matrix $W\in\mathbb{R}^{s\times s}$ as follows:
\begin{equation}\label{A(x,W)}
  \mathcal{A}(x,W):=\left[
		\begin{array}{ccc}
			\nabla_{yy}^2 \mathcal{L}(x;y(x),\mu(x),\xi(x))&\mathcal{J}_y h(x,y(x))^T&\mathcal{J}_y g(x,y(x))^T\\
             \mathcal{J}_y h(x,y(x))&0&0\\
			(I-W)\mathcal{J}_y g(x,y(x))&0&-W
		\end{array}\right]
\end{equation}
and
\begin{align}\label{H(x,W)}
  \mathcal{H}(x,W):=& \mathcal{A}(x,W)^{-1}
  \left(
		\begin{array}{c}
			\nabla_{yx}^2 \mathcal{L}(x;y(x),\mu(x),\xi(x))\\
\mathcal{J}_x h(x,y(x))\\
			(I-W)\mathcal{J}_x g(x,y(x))
		\end{array}\right).
\end{align}

%Denote set-valued mappings $\mathbb{A}_B(x):=\{\mathcal{A}(x,W):W\in \partial_B\Pi_{\mathbb{R}^s_{-} }(g(x,y(x))+\xi(x))\}$ and
%$\mathbb{A}_C(x):=\{\mathcal{A}(x,W):W\in \partial \Pi_{\mathbb{R}^s_{-} }(g(x,y(x))+\xi(x))\}$.

\begin{lemma}\label{lemma Assump A y()}
%Let $(x^*, y^*) \in \mathbb{R}^n \times \mathbb{R}^m$ be a point around which $f,g,h$ are  continuously %differentiable and twice continuously differentiable with respect to the variable $y$.
 Suppose that assumptions   \textnormal{(A2)} and \textnormal{(A4)}  hold. Then there exist positive constants
$\delta_3<\delta_2$ and $\varepsilon_3<\varepsilon_2$, and a locally Lipschitz continuous mapping  $
(y,\mu,\xi):\mathbb{B}_{\delta_3}(x^*)\to\mathbb{B}_{\varepsilon_3}(y^*)\times\mathbb{B}_{\varepsilon_3}(\mu^*)
\times\mathbb{B}_{\varepsilon_3}(\xi^*) $ such that $(y(x^*),\mu(x^*),\xi(x^*))=(y^*,\mu^*,\xi^*)$ and $(y(x),\mu(x),\xi(x))$ is the locally unique solution pair
 satisfying the KKT conditions  of the lower-level problem $(P_x)$, i.e.,
 \begin{align}\label{Fkkt}
           F_{KKT}(x):= & \left[\begin{array}{c}
                                \nabla_y \mathcal{L}(x;y(x) ,\mu(x) ,\xi(x))\\
                                 h(x,y(x))\\
                                  g(x,y(x))-\Pi_{\mathbb{R}^s_{-}}(g(x,y(x))+\xi(x))
                              \end{array}\right]=0,
         \end{align}
          and $(y(x),\mu(x),\xi(x))$   is second-order directionally differentiable and second-order gph-regular
for $x\in \mathbb{B}_{\delta_3}(x^*)$. Moreover, the Clarke generalized Jacobian of $(y(\cdot),\mu(\cdot),\xi(\cdot))$ at $x$ satisfies
\begin{align}\label{partial y mu xi}
     \partial \left(\begin{array}{c}
                       y \\
                      \mu \\
                       \xi
                     \end{array}\right)(x)\in &\left\{-\mathcal{H}(x,W):W\in \partial \Pi_{\mathbb{R}^s_{-} }(g(x,y(x))+\xi(x))\right\},
\end{align}

when $x\in \mathbb{B}_{\delta_3}(x^*)$.
\end{lemma}
\begin{proof}
  The main results follow from  \cite[Theorems 2.1 and 4.1 and Corollary 2.2]{Robinson 80} and   \cite[Proposition 2.4]{Daiyh and Zhanglw}. It follows from \cite[Proposition 2.3]{Daiyh and Zhanglw} that   the set $\{\mathcal{A}(x,W):W\in \partial \Pi_{\mathbb{R}^s_{-} }(g(x,y(x))+\xi(x))\}$ is of maximal rank, then the second-order directional differentiability and second-order gph-regularity are established via  Corollary \ref{cor implicit function outer second-order gph-regular}.
\end{proof}

\subsection{First-order necessary optimality conditions in primary form}\label{sec first-order optimality
condition by directional}

In this subsection, we derive first-order necessary optimality conditions  for (SP) and (FP) in primary form under different assumptions  by exploiting directional derivatives. Define $\Phi:=\{(x,y)\in \mathbb{R}^n \times \mathbb{R}^m:H(x,y)=0, G(x,y)\le0\}$ and denote the feasible sets of (SP) and (FP)  by $\Omega_1$ and $\Omega_2$, respectively.
\begin{theorem}\label{thm FONC for SP}
Suppose that assumptions  \textnormal{(A1)}--\textnormal{(A3)}  hold  and $(x^*,y^*)$ is a bi-local solution of (BP). If the MSCQ for $\Omega_1$ holds at $x^*$,
    then $d_x=0$ is an optimal solution of the problem $(SP_{d_x})$
       \begin{subequations}\label{linearied cone SP}
         \begin{align}
   \min_{d_x} ~ & \nabla_x F(x^*, y^*)^{T} d_x+\nabla_y F(x^*, y^*)^{T}y'(x^*;d_x)\notag\\
   {\rm s.t.}~&\nabla_x H_i(x^*, y^*)^{T} d_x+\nabla_y H_i(x^*, y^*)^{T} y'(x^*;d_x) = 0,
                  \quad  i\in[p], \\
        &\nabla_x G_i(x^*, y^*)^{T} d_x+\nabla_y G_i(x^*, y^*)^{T} y'(x^*;d_x)  \leq 0,
         \quad  i \in I_{G}.
   \end{align}
       \end{subequations}
\end{theorem}
\begin{proof}
It follows from Proposition \ref{prop-tangent cone=linearized} and (\ref{composite chain rule h''}) that we derive
that the tangent cone $\mathcal{T}_{\Omega_1}(x^*)$ coincides with (\ref{linearied cone SP}).
  Since (SP) is a nonsmooth optimization problem and functions involved are Lipschitz continuous and
  directionally differentiable, then it follows from Theorem \ref{thm FONC for nonsmooth} that we can derive the
  consequence.
\end{proof}

 In the rest of the subsection, we demonstrate that  the  first-order optimality conditions of $(SP)$ and $(FP)$ are equivalent under assumptions \textnormal{(A2)} and \textnormal{(A4)}. First we show that the nonsmooth Mangasarian-Fromovitz constraint qualification holds for the problem (SP) if and only if the one holds for the problem (FP). The generalized MFCQ (GMFCQ) was first introduced by Hiriart-Urruty \cite{Hiriart-Urruty}.

 It follows from (\ref{partial y mu xi}) that $\partial y(x)\subseteq\left\{ -[ I_m~ 0~ 0]\mathcal{H}(x,W):W\in \partial \Pi_{\mathbb{R}^s_{-}}(g(x,y(x))+\xi(x))  \right\}$.
 We say the GMFCQ holds at $x^*$ for $\Omega_1$
 if for any $ W^*\in
\partial\Pi_{\mathbb{R}_{-}^s}(g(x^*,y^*)+\xi^*)$, the following conditions hold:
\begin{itemize}
  \item [(i)]  $\mathcal{J}_xH(x^*,y^*)-\mathcal{J}_yH(x^*,y^*)[\begin{array}{ccc}
                                                              I_m & 0 & 0
                                                            \end{array}]\mathcal{H}(x^*,W^*)$ has full row rank.
  \item [(ii)]There exists $d_x\in \mathbb{R}^n$ such that
  \begin{align} \label{MFCQ-b}
   &    \left(\mathcal{J}_xH(x^*,y^*)-\mathcal{J}_yH(x^*,y^*)[\begin{array}{ccc}
                                                              I_m & 0 & 0
                                                            \end{array}]\mathcal{H}(x^*,W^*)\right)d_x=0, \\
\label{MFCQ-c}&  \left(\nabla_xG_i(x^*,y^*)^T-\nabla_yG_i(x^*,y^*)^T[\begin{array}{ccc}
                                                              I_m & 0 & 0
                                                            \end{array}]\mathcal{H}(x^*,W^*)\right)d_x<0,\quad
                                                            i\in I_{G},
  \end{align}
   where $I_G:=\{i\in[q]:G_i(x^*,y^*)=0\}$.
\end{itemize}

\begin{proposition}\label{prop MFCQ of SP FP}
 Suppose that assumptions \textnormal{(A2)} and \textnormal{(A4)} hold. Then
the GMFCQ holds at $x^*$ for $\Omega_1$ if and only if  the GMFCQ holds for $\Omega_2$ at $(x^*,y^*,\mu^*,\xi^*)$.
	\end{proposition}
\begin{proof} The GMFCQ holds for $\Omega_2$ at $(x^*,y^*,\mu^*,\xi^*)$, if for any $ W^*\in
\partial\Pi_{\mathbb{R}_{-}^s}(g(x^*,y^*)+\xi^*)$, the following conditions hold:
\begin{itemize}
  \item[(a)] The matrix
  \begin{align}%\label{A}
           &  A:=\left[\begin{array}{cccc}
                     \mathcal{J}_xH(x^*,y^*) & \mathcal{J}_yH(x^*,y^*) & 0 & 0 \\
                    \nabla_{yx}^2 \mathcal{L}(x^*;y^*,\mu^*,\xi^*)
                    &\nabla_{yy}^2 \mathcal{L}(x^*;y^*,\mu^*,\xi^*)
                    &\mathcal{J}_y h(x^*,y^*)^T &\mathcal{J}_y g(x^*,y^*)^T \\
                     \mathcal{J}_x h(x^*,y^*) & \mathcal{J}_y h(x^*,y^*)&0&0 \\
                     (I-W^*)\mathcal{J}_x g(x^*,y^*) & (I-W^*)\mathcal{J}_y g(x^*,y^*)&0& -W^*
                   \end{array}
  \right]\notag
  \end{align}
  has full row rank.
  \item[(b)]There is $d= (d_x,d_y,d_{\mu},d_{\xi})\in ker A $ such that
      \begin{eqnarray}
    &&  \nabla_xG_i(x^*,y^*)^Td_x+  \nabla_yG_i(x^*,y^*)^T d_y  <0,\quad i\in I_G. \label{mfg}
  \end{eqnarray}
\end{itemize}
We prove the equivalence by the following two steps.

(1) For any $ W^*\in \partial\Pi_{\mathbb{R}_{-}^s}(g(x^*,y^*)+\xi^*)$, let
$$B:=\left[\begin{array}{cc}
                I_n & 0  \\
                 -\mathcal{H}(x^*,W^*)&I_{m+r+s}
              \end{array}\right].$$
Since the matrix $B$ is nonsingular, then $rank(A)=rank(AB)$.
                By (\ref{H(x,W)}), we have that
 \begin{align*}
   &AB=\left[\begin{array}{cc}
                \mathcal{J}_xH(x^*,y^*)-
                V\mathcal{H}(x^*,W^*)&V\\
                0&\mathcal{A}(x^*,W^*)
              \end{array}\right],
              \end{align*}
  where $V:=[\begin{array}{ccc}
       \mathcal{J}_yH(x^*,y^*) &0 & 0
        \end{array}]=\mathcal{J}_yH(x^*,y^*) [\begin{array}{ccc}
          I_m &0 & 0
            \end{array}]$ and $ \mathcal{A}(x,W )$ is defined by (\ref{A(x,W)}).
 It follows from \cite[Proposition 2.2]{Daiyh and Zhanglw} that $\mathcal{A}(x^*,W^*)$  is nonsingular,
  then $A$ has full row rank if and only if
  $\mathcal{J}_xH(x^*,y^*)-\mathcal{J}_yH(x^*,y^*)[\begin{array}{ccc}
          I_m &0 & 0
            \end{array}] \mathcal{H}(x^*,W^*)$ has full row rank.

(2) Suppose there exists $d= (d_x,d_y,d_{\mu},d_{\xi})\in ker A$, i.e.,
  \begin{eqnarray}
        &\mathcal{J}_xH(x^*,y^*)d_x+\mathcal{J}_yH(x^*,y^*)d_y=0,\label{4-2-1}\\
    & \left[ \begin{array}{c}
           \nabla_{yx}^2 \mathcal{L}(x^*;y^*,\mu^*,\xi^*)
                     \\
          \mathcal{J}_x h(x^*,y^*)  \\
       (I-W^*)\mathcal{J}_x g(x^*,y^*)
         \end{array} \right]d_x +\mathcal{A}(x^*,W^*)\left[\begin{array}{c}
       d_y\\
       d_{\mu}\\
       d_{\xi}
              \end{array}
         \right]=0.\label{4-2}
 \end{eqnarray}
Then the condition $ (\ref{4-2})$ becomes
 \begin{eqnarray}
&&\left[\begin{array}{c}
       d_y\\
       d_{\mu}\\
       d_{\xi}
              \end{array}
         \right]=-\mathcal{A}(x^*,W^*)^{-1}\left[ \begin{array}{c}
           \nabla_{yx}^2 \mathcal{L}(x^*;y^*,\mu^*,\xi^*)
                     \\
          \mathcal{J}_x h(x^*,y^*)  \\
       (I-W^*)\mathcal{J}_x g(x^*,y^*)
         \end{array} \right]d_x=-\mathcal{H}(x^*,W^*)d_x.\label{4-3}
 \end{eqnarray}
Replacing $d_y$ by $ -[\begin{array}{ccc}
                        I_m & 0 & 0
                      \end{array}]\mathcal{H}(x^*,W^*)d_x$ into the conditions $(\ref{4-2-1})$ and (\ref{mfg}), we
                      derive the conditions (\ref{MFCQ-b}) and (\ref{MFCQ-c}).
 It shows that if the GMFCQ holds at $(x^*,y^*,\mu^*,\xi^*)$ for $\Omega_2$, then
 $\mathcal{J}_xH(x^*,y^*)-\mathcal{J}_yH(x^*,y^*)[\begin{array}{ccc}
          I_m &0 & 0
            \end{array}] \mathcal{H}(x^*,W^*)$ has full row rank and (\ref{MFCQ-b})-(\ref{MFCQ-c}) hold.

 Conversely, if there exists $d_x\in \mathbb{R}^n$ such that (\ref{MFCQ-b})-(\ref{MFCQ-c}) hold, by setting $(d_y,
 d_{\mu}, d_{\xi})$ as in (\ref{4-3}), we derive that (\ref{mfg})-(\ref{4-2}) hold. Then the GMFCQ holds at
 $(x^*,y^*,\mu^*,\xi^*)$ for $\Omega_2$.
We complete the proof.
\end{proof}
%\begin{remark}
%Notice that $\partial y(x^*)$ has the upper estimate $\{-[\begin{array}{ccc}
%                             I_m&0&0
%                           \end{array}]\mathcal{H}(x^*,W^*):W^*\in \partial \Pi_{\mathbb{R}^s_{-}}(g(x^*,y(x^*))+\xi(x^*))\}$ from \cite[Proposition 2.4]{Daiyh and Zhanglw}.
%Then the GMFCQ holds at $x^*$ for (SP) are slightly weaker than the conditions (i) and (ii) of Proposition \ref{prop MFCQ of SP FP}. For the sake of clarity, we take them as the GMFCQ at $x^*$ for the problem (SP).
%\end{remark}
It follows from  \cite[Theorem 4.2]{Hiriart-Urruty} and \cite[Corollary 5 of Theorem 6.5.2]{Clarke} that the GMFCQ is stronger than MSCQ.
From Proposition \ref{prop MFCQ of SP FP}, we derive the following equivalent conditions of the  first-order optimality conditions.
\begin{theorem}\label{thm FONC for FP}
 Suppose that assumptions \textnormal{(A2)} and \textnormal{(A4)} hold and $(x^*,y^*)$ is a bi-local solution of (BP). If the GMFCQ for $\Omega_2$ holds at $(x^*,y^*,\mu^*,\xi^*)$,
   then $d=0$ is an optimal solution of the problem  $(FP_{d})$
       \begin{subequations}\label{linearied cone FP}
         \begin{align}
         \label{linearied FP'}
   \min_{d}  ~& \nabla F(x^*, y^*)^{T} d_{(x,y)}\\
   \label{linearied FP b}
   {\rm s.t.}~&\nabla H_i(x^*, y^*)^{T} d_{(x,y)} = 0, \quad  i\in[p],  \\
   \label{linearied FP c}
 &\nabla G_i(x^*, y^*)^{T} d_{(x,y)}  \leq 0,\quad  i \in I_{G},\\
\label{linearied FP d}
        &\nabla_{yx}^2 \mathcal{L}(x^*;y^*,\mu^*,\xi^*)d_x+\nabla_{yy}^2 \mathcal{L}(x^*;y^*,\mu^*,\xi^*)d_y
        +\mathcal{J}_y h(x^*,y^*)^Td_{\mu}+\mathcal{J}_yg(x^*,y^*)^Td_{\xi}=0, \\
\label{linearied FP e}
        &   \nabla h_i(x^*,y^*)^Td_{(x,y)}=0, \quad  i\in[r],  \\
  \label{linearied FP f}
  &\nabla g_i(x^*,y^*)^Td_{(x,y)}-\Pi_{\mathbb{R}_{-}}'(g_i(x^*,y^*)+\xi_i^*;\nabla
             g_i(x^*,y^*)^Td_{(x,y)}+d_{\xi_i})=0,~i\in[s].
   \end{align}
       \end{subequations}

 Furthermore, the problems $(SP_{d_x})$ and $(FP_{d})$ are equivalent.
\end{theorem}

\begin{proof}
From \cite[Remark 3.6]{Liu 25}, we know that $x^*$ is a local solution of (SP), and $(x^*,y^*,\mu^*,\xi^*)$ is a local solution of (FP).
It follows from Proposition \ref{prop-tangent cone=linearized} and (\ref{composite chain rule h''}) that we derive that  the tangent cone $\mathcal{T}_{\Omega_2}(x^*,y^*,\mu^*,\xi^*)$ coincides with (\ref{linearied FP b})-({\ref{linearied FP f}}).
  Since  (FP) is a nonsmooth optimization problem  and functions involved are Lipschitz continuous and
  directionally differentiable, then it follows from Theorem \ref{thm FONC for nonsmooth} that we can derive that $d=0$ is an optimal solution of the problem  $(FP_{d})$.

From (\ref{Fkkt}), we obtain that $F_{KKT}'(x;d_x)=0$ for any $d_x$.
Then for any $d_x\in \mathcal{T}_{\Omega_1}(x^*)$, it follows from the chain rule in  (\ref{composite chain rule h''}) that  (\ref{linearied FP b})-(\ref{linearied FP f}) hold. That is
 $$(d_x,y'(x^*;d_x),\mu'(x^*;d_x),\xi'(x^*;d_x))\in \mathcal{T}_{\Omega_2}(x^*,y^*,\mu^*,\xi^*).$$

%Furthermore, % equations (\ref{linearied FP c})-(\ref{linearied FP e3}) is equivalent to
%since $F^{KKT}(x,y(x),\mu(x),\xi(x))=0$, then
% \begin{align}%\label{F^KKT'(x;d_x,w_x)}
%\left\{\begin{array}{l}
%         \nabla_{yx}^2 \mathcal{L}(x;y(x),\mu(x),\xi(x))d_x+\nabla_{yy}^2
%         \mathcal{L}(x;y(x),\mu(x),\xi(x))y'(x;d_x)
%         \\
%         \qquad+\mathcal{J}_y h(x,y(x))^T\mu'(x;d_x) +\mathcal{J}_yg(x,y(x))^T\xi'(x;d_x)=0, \\
%         \mathcal{J}_x h(x,y(x))d_x+ \mathcal{J}_y h(x,y(x))y'(x;d_x)=0, \\
%        \mathcal{J}_x g (x,y(x))d_x+ \mathcal{J}_y g(x,y(x))y'(x;d_x)-\Pi_{\mathbb{R}_{-}^s}'(g(x,y(x))+\xi(x);
%        \\
%         \qquad\mathcal{J}_x g(x,y(x))d_x+ \mathcal{J}_y g(x,y(x))y'(x;d_x)+\xi'(x;d_x))=0.
%       \end{array}\right.
%  \end{align}
% Thus we derive that
%\begin{align}\label{sol}
%    & (d_y,d_{\mu},d_{\xi})=(y'(x^*;d_x),\mu'(x^*;d_x),\xi'(x^*;d_x))
%\end{align}
%is a solution of (\ref{linearied FP d})-(\ref{linearied FP f}).
  Calculating the Clarke generalized Jacobian of (\ref{linearied FP d})-(\ref{linearied FP f}) at every point in $\mathbb{R}^{m+r+s}$ with respect to variables
  $d_y,d_{\mu},d_{\xi}:$
  \begin{equation*}
  \left\{\left[
		\begin{array}{ccc}
			\nabla_{yy}^2 \mathcal{L}(x^*;y^*,\mu^*,\xi^*)&\mathcal{J}_y h(x^*,y^*)^T&\mathcal{J}_y g(x^*,y^*)^T\\
             \mathcal{J}_y h(x^*,y^*)&0&0\\
			(I-W^*)\mathcal{J}_y g(x^*,y^*)&0&-W^*
		\end{array}\right]:W^*\in \partial \Pi_{\mathbb{R}^s_{-} }(g(x^*,y^*)+\xi^*)\right\}:=\mathbb{A}_C(x^*).
\end{equation*}
It follows from \cite[Proposition 2.2]{Daiyh and Zhanglw} that every element in $ \mathbb{A}_C(x^*)$ is nonsingular.
From the Clarke implicit function theorem in \cite[Corollary of Theorem 7.1.1]{Clarke} and \cite[Proposition 2.6.5]{Clarke},
%every element in $ \mathbb{A}_C(x^*)$ is nonsingular. It follows from Lemma \ref{lemma implicit func is unique}
it follows that there exists only one solution $(d_x,y'(x^*;d_x),\mu'(x^*;d_x),\xi'(x^*;d_x))$ satisfying (\ref{linearied FP d})-(\ref{linearied FP f}).
Thus for any $d=(d_x, d_y,d_{\mu},d_{\xi})\in \mathcal{T}_{\Omega_2}(x^*,y^*,\mu^*,\xi^*)$, we must have that $(d_y,d_{\mu},d_{\xi})=(y'(x^*;d_x),\mu'(x^*;d_x),\xi'(x^*;d_x))$.
Then from (\ref{linearied FP b})-(\ref{linearied FP c}), $d_x\in \mathcal{T}_{\Omega_1}(x^*)$.

%it holds that \begin{align}\label{d_y=y'(x^*;d_x),d_mu,d_xi}
%  &d=(d_x,y'(x^*;d_x),\mu'(x^*;d_x),\xi'(x^*;d_x))\in \mathcal{T}_{\Omega_2}(x^*,y^*,\mu^*,\xi^*),
%\end{align}
%thus $d_x\in \mathcal{T}_{\Omega_1}(x^*)$.
%Therefore we can verity that $d_x\in \mathcal{T}_{\Omega_1}(x^*)$ if and only if (\ref{d_y=y'(x^*;d_x),d_mu,d_xi}) holds.
It is straightforward to derive that problems  $(SP_{d_x})$ and $(FP_d)$ are equivalent.
We complete the proof.
\end{proof}

%{\color{blue}Notice that (\ref{linearied FP f}) can be rewrited as
%\begin{align*}
%     &\nabla g_i(x^*,y^*)^Td_{(x,y)}=0, \quad i\in I_{x^*}^+(y^*),\qquad d_{\xi_i}=0,\quad i\in
%         I_{x^*}^-(y^*),\\
%  &0\ge \nabla g_i(x^*,y^*)^Td_{(x,y)}\bot d_{\xi_i}\ge0,\quad i\in I_{x^*}^0(y^*),
%\end{align*}
% where      $I_{x^*}^+(y^*)=\{i\in[s]:g_i(x^*,y^*)=0,\xi_i>0\}$,
% $I_{x^*}^0(y^*):=\{i\in[s]:g_i(x^*,y^*)=0,\xi_i=0\}$ and $I_{x^*}^-(y^*):=\{i\in[s]:g_i(x^*,y^*)<0,\xi_i=0\}$.}

%\begin{remark}
%From \cite[Proposition  2.4]{Daiyh and Zhanglw} and \cite[Lemma 2.2]{Qi}, we can know that the upper estimate of Clarke generalized Jacobian of
%$y'(x^*;\cdot)$ at $0$ is $\partial y(x^*)$. By the necessary optimality conditions of nonsmooth problem
%under which the GMFCQ holds at a local solution, we can derive that the dual form of first-order necessary optimality conditions, which is same as the following Theorems \ref{thm FONC for SP} and \ref{thm FONC for FP}.
%\end{remark}

\subsection{First-order necessary optimality conditions  in dual form}\label{sec: Optimality conditions by Clarke subdifferential}
In this subsection, we establish the dual form of first-order necessary optimality conditions for
 (BP) at a bi-local solution  under assumptions \textnormal{(A2)} and \textnormal{(A4)}.
 % by means of the ones for the problems (SP) and (FP) at a local solution.
The  Lagrangian functions of (SP) and (FP) are defined by
\begin{eqnarray}
   &&L^{SP}(x;\lambda_H,\lambda_G)=
   F(x,y(x))+\lambda_H^TH(x,y(x))+\lambda_G^TG(x,y(x)),\nonumber\\
   &&L^{FP}(x,y,\mu,\xi;\lambda)= F(x,y)+\lambda_H^TH(x,y)+\lambda_G^TG(x,y)
    +\lambda_{\mathcal{L}}^T\nabla_{y} \mathcal{L}(x;y,\mu,\xi)\nonumber\\
    &&~~~~~~~~~~~~~~~~~~~~~~~~~~~~~~+\lambda_h^Th(x,y)+\lambda_g^T(g(x,y)-\Pi_{\mathbb{R}_{-}^s}(g(x,y)+\xi)),\nonumber
 \end{eqnarray}
where $\lambda:=(\lambda_H,\lambda_G,\lambda_{\mathcal{L}},\lambda_h,\lambda_g)$, respectively.
For clarity of the notation, we define that
\begin{align*}
  & L(x,y;\lambda_H,\lambda_G)= F(x,y)+\lambda_H^TH(x,y)+\lambda_G^TG(x,y).
\end{align*}
%The following results give the dual form of first-order optimality conditions of bi-local solutions and show the equivalence between the optimality conditions of local solutions for (SP) and (FP).

\begin{theorem}\label{theorem dual First-order NOC for SP}
Suppose that assumptions \textnormal{(A2) and \textnormal{(A4)} hold} and  $(x^*,y^*)$ is a bi-local solution of (BP). If the GMFCQ holds at $x^*$ for $\Omega_1$, then there exist  $(\lambda_H^*,\lambda_G^*)\in
      \mathbb{R}^{p}\times\mathbb{R}^{q}$ and $ W^*\in
      \partial\Pi_{\mathbb{R}_{-}^s}(g(x^*,y^*)+\xi^*)$ such that
      \begin{align}
        &   \nabla_{x} L (x^*,y^*;\lambda^*_H,\lambda^*_G)-\mathcal{H}(x^*,W^*)^T\left[\begin{array}{c}
         \nabla_{y} L (x^*,y^*;\lambda^*_H,\lambda^*_G)\\0\\0
                  \end{array}\right]=0,\notag\\
         & H(x^*,y^*)= 0,\label{3-1}\\
        &0\le \lambda_G^*\bot G(x^*,y^*)\le 0.\notag
      \end{align}
\end{theorem}
\begin{proof}
  Since $(x^*,y^*)$ is a bi-local solution of (BP), we  know that $x^*$ is a local solution of (SP) from
  \cite[Remark 3.6]{Liu 25}.
  It follows from
the nonsmooth Fritz John-type multiplier rule by Clarke \cite[Theorem 6.1.1]{Clarke}
 and the GMFCQ for $\Omega_1$ that %there is no nonzero abnormal multiplier for Fritz John necessary optimality
%  conditions for (SP), i.e.,
there exists
$(\lambda_H^*,\lambda_G^*)\in  \mathbb{R}^{p}\times\mathbb{R}^{q}_{+}$, not all zero such
that $\langle \lambda_{G}^*, G(x^*,y^*)\rangle= 0$ and
\begin{align*}
 0 \in \partial_xL^{SP}(x^*,\lambda_H^*,\lambda_G^*)\subseteq
  \{&\nabla_xL(x^*,y^*;\lambda_H^*,\lambda_G^*)-\mathcal{H}(x^*,W^*)^T[\begin{array}{ccc}
                             I_m&0&0
                           \end{array}]^T\nabla_yL(x^*,y^*;\lambda_H^*,\lambda_G^*)\\
                           &:W^*\in \partial \Pi_{\mathbb{R}^s_{-}}(g(x^*,y^*)+\xi^*)\},
\end{align*}
where the last inclusion relation holds by the chain rule of \cite[Theorem 10.20]{Clarke 13} and \cite[Proposition 2.4]{Daiyh and Zhanglw}.
The proof is completed.
\end{proof}

\begin{theorem}\label{theorem dual First-order NOC for FP}
 Suppose that assumptions \textnormal{(A2)} and \textnormal{(A4)} hold and  $(x^*,y^*)$ is a bi-local solution of (BP).  If the GMFCQ holds at $(x^*,y^*,\mu^*,\xi^*)$ for $\Omega_2$, then there exist
      $\lambda^*=(\lambda_H^*,\lambda_G^*,\lambda_{\mathcal{L}}^*,\lambda_h^*,\lambda_g^*)
  \in\mathbb{R}^{p}\times\mathbb{R}^{q}\times\mathbb{R}^{m}
  \times\mathbb{R}^{r}\times\mathbb{R}^{s}$ and $ W^*\in \partial\Pi_{\mathbb{R}_{-}^s}(g(x^*,y^*)+\xi^*)$ such
  that
   \begin{align}\label{3-2}
&\nabla_{x}L(x^*,y^*;\lambda_H^*,\lambda_G^*)+\nabla_{xy}^2 \mathcal{L}(x^*;y^*,\mu^*,\xi^*)
\lambda_{\mathcal{L}}^* +\mathcal{J}_x h(x^*,y^*)^T\lambda_h^*+ \mathcal{J}_x
g(x^*,y^*)^T(I-W^*)\lambda_g^*=0,\notag\\
&  \nabla_{y}L(x^*,y^*;\lambda_H^*,\lambda_G^*)+\nabla_{yy}^2 \mathcal{L}(x^*;y^*,\mu^*,\xi^*)
\lambda_{\mathcal{L}}^* +\mathcal{J}_y h(x^*,y^*)^T\lambda_h^*+ \mathcal{J}_y
g(x^*,y^*)^T(I-W^*)\lambda_g^*=0,\notag\\
   & \mathcal{J}_y h(x^*,y^*)\lambda_{\mathcal{L}}^*=0,\notag \\
  &\mathcal{J}_y g(x^*,y^*)\lambda_{\mathcal{L}}^*-W^*\lambda_g^*=0,\notag\\
  &H(x^*,y^*)=0,\\
  &0\le \lambda_G^*\bot G(x^*,y^*)\le 0,\notag
%  &\nabla_{y} \mathcal{L}(x^*;y^*,\mu^*,\xi^*)=0,\notag\\
%  &h(x^*,y^*)=0,\notag\\
%&g(x^*,y^*)-\Pi_{\mathbb{R}_{-}^s}(g(x^*,y^*)+\xi^*)=0.\notag
\end{align}
and  (\ref{kkt}) hold. Furthermore, (\ref{3-1}) are equivalent to (\ref{3-2}) and  (\ref{kkt}).
\end{theorem}
\begin{proof}
 Since assumptions \textnormal{(A2)} and \textnormal{(A4)} hold and $(x^*,y^*)$ is a bi-local solution of (BP), then (\ref{kkt}) holds and
  $(x^*,y^*,\mu^*,\xi^*)$ is a local solution of (FP) from \cite[Remark 3.6]{Liu 25}.
 %i.e., the last three equations of (\ref{3-2}) holds.
 It follows from \cite[Theorem 6.1.1]{Clarke} and
  GMFCQ holds for $\Omega_2$ at $(x^*,y^*,\mu^*,\xi^*)$ that
 % From Proposition \ref{prop MFCQ of SP FP}, we readily know that the GMFCQ holds at $x^*$ for (SP) if and only
%  if the GMFCQ holds at $(x^*,y^*,\mu^*,\xi^*)$ for (FP).
%  And the equivalence of (i) and (ii) in Theorem \ref{theorem (SP) FJ conditions} reduces to
%   (\ref{3-1}) and (\ref{3-2}) with $r^*=1$.
%  We complete the proof.
  there exists
$(\lambda_H^*,\lambda_G^*,\lambda_{\mathcal{L}}^*,\lambda_h^*,\lambda_g^*)
  \in\mathbb{R}^{p}\times\mathbb{R}_{+}^{q}\times\mathbb{R}^{m}
  \times\mathbb{R}^{r}\times\mathbb{R}^{s}$, not all zero, such that
   $\langle \lambda_{G}^*, G(x^*,y^*)\rangle= 0$ and
\begin{align*}
 0 \in& \partial_{(x,y,\mu,\xi)} L^{FP}(x^*,y^*,\mu^*,\xi^*;\lambda^*)\\
 \subseteq& \left\{\left[\begin{array}{c}
                \nabla_{x}L(x^*,y^*;\lambda_H^*,\lambda_G^*)
                  \\
               \nabla_{y}L(x^*,y^*;\lambda_H^*,\lambda_G^*)
                  \\
               0 \\
                0
              \end{array}\right]\right. +\left[\begin{array}{c}
                            \nabla_{xy}^2 \mathcal{L}(x^*;y^*,\mu^*,\xi^*)\lambda_{\mathcal{L}}^* \\
                            \nabla_{yy}^2 \mathcal{L}(x^*;y^*,\mu^*,\xi^*)\lambda_{\mathcal{L}}^* \\
                                                      \mathcal{J}_y h(x^*,y^*)\lambda_{\mathcal{L}}^* \\
                                                      \mathcal{J}_y g(x^*,y^*)\lambda_{\mathcal{L}}^*
                                                    \end{array}\right]
                                                    +\left[\begin{array}{c}
                         \mathcal{J}_x h(x^*,y^*)^T\lambda_h^* \\
                         \mathcal{J}_y h(x^*,y^*)^T\lambda_h^* \\
                         0 \\
                         0
                       \end{array}\right]\notag\\
              &~~~  +      \left.\left[\begin{array}{c}
                      \mathcal{J}_x
g(x^*,y^*)^T(I-W^*)\lambda_g^* \\
                      \mathcal{J}_y
g(x^*,y^*)^T(I-W^*)\lambda_g^*\\
                      0\\
                    -W^*\lambda_g^*
                    \end{array}\right]:W^*\in \partial \Pi_{\mathbb{R}_{-}^{s}}(g(x^*,y^*)+\xi^*)
              \right\},
\end{align*}
where the last inclusion relation holds by Proposition 2.3.3 (Finite Sums) and Theorem 2.3.9 (Chain Rule $\text{\uppercase\expandafter{\romannumeral1}}$)
of \cite{Clarke}.
Therefore (\ref{3-2}) holds.

The proof of the equivalence of    the first-order conditions is similar to the one in  \cite[Corollary 4.1]{Liu 25}. We complete the proof.
\end{proof}

\subsection{Second-order necessary and sufficient optimality conditions}\label{sec second-order optimality
condition by directional}
In this subsection, we derive the second-order optimality conditions for (BP) under
different assumptions  by exploiting directional derivatives.
Under the  MSCQ, the critical cones at $x^*$ for (SP) and
at
$(x^*,y^*,\mu^*,\xi^*)$ for (FP) can be represented by
\begin{align*}
  & \mathcal{C}_{\Omega_1}(x^*)=\mathcal{T}_{\Omega_1}(x^*)\cap \{d_{x}\in\mathbb{R}^n:\nabla_x F(x^*, y^*)^{T}
  d_x+\nabla_y F(x^*, y^*)^{T} y'(x;d_x)  \leq 0\},\\
  &\mathcal{C}_{\Omega_2}(x^*,y^*,\mu^*,\xi^*)=\mathcal{T}_{\Omega_2}(x^*,y^*,\mu^*,\xi^*)\cap
  \{d\in\mathbb{R}^{n+m+r+s}:\nabla F(x^*, y^*)^{T}d_{(x,y)}  \leq 0\},
\end{align*}
respectively.
%When GMFCQs hold, the tangent cone $\mathcal{T}_{\Omega_1}(x^*)$ coincides with the set of all $d_x$
%satisfying (\ref{linearied cone SP}) and $\mathcal{T}_{\Omega_2}(x^*,y^*,\mu^*,\xi^*)$  coincides with the set of
%all $d$ satisfying (\ref{linearied FP b})-(\ref{linearied FP f}).

%In this section, we assume $(x^*,y^*)\in \Phi$ be a point around which $f,g,h$ are twice
%continuously differentiable and thrice continuously differentiable with respect to variable $y$ and $F,G,H$ are
%twice continuously differentiable.
\begin{theorem}[Second-order necessary and sufficient optimality conditions]\label{thm SOSC for SP}
Suppose that assumptions \textnormal{(A1)}--\textnormal{(A3)} hold and the MSCQ for $\Omega_1$ holds at $x^*$.
\begin{itemize}
  \item[(i)]If $(x^*,y^*)$ is a bi-local solution of (BP), then for every $d_x\in \mathcal{C}_{\Omega_1}(x^*)$ and all $w_x\in
      \mathbb{R}^n$ satisfying
      \begin{align}\label{second-order tangent set of Omega_1}
 &   \tilde{d}^T\nabla^2 H_i(x^*, y^*)\tilde{d} +\nabla H_i(x^*, y^*)^{T}\tilde{w}=0,~ i\in[p], \\
 &    \tilde{d}^T\nabla^2 G_i(x^*, y^*)\tilde{d}  +\nabla G_i(x^*, y^*)^{T}\tilde{w}\le0,~ i\in I_G^1,\notag
\end{align}
where $\tilde{d}:=[d_x; y'(x^*;d_x)]$, $\tilde{w}:=[w_x;y''(x^*;d_x,w_x)]$  and $I_G^1:=\{i\in I_G:\nabla_x G_i(x^*, y^*)^{T} d_x+\nabla_y G_i(x^*, y^*)^{T} y'(x^*;d_x)=0\}$,
       it follows that
   \begin{align}\label{align F''}
        &\tilde{d}^T\nabla^2 F(x^*, y^*)\tilde{d}+\nabla F(x^*, y^*)^{T}\tilde{w}\ge0.\notag
      \end{align}
  \item[(ii)] Assume that for every $d_x\in\mathcal{T}_{\Omega_1}(x^*)$, we
        have
 \begin{align*}
   & \nabla_x F(x^*, y^*)^{T} d_x+\nabla_y F(x^*, y^*)^{T} y'(x^*;d_x)\ge 0
 \end{align*} and  the optimization problem
   \begin{align}%\label{p SP F''}
    \min_{w_x}~&\tilde{d}^T\nabla^2 F(x^*, y^*)\tilde{d}+\nabla F(x^*, y^*)^{T}\tilde{w}\notag \\
     {\rm s.t.}~&w_x ~{\rm satisfying}~ (\ref{second-order tangent set of Omega_1})\notag,
   \end{align}
possesses a positive optimal objective value for every $d_x\in \mathcal{C}_{\Omega_1}(x^*)\backslash \{0\}$.
Then there exist $\delta_2' \in(0, \delta_2), \varepsilon_2' \in(0, \varepsilon_2)$ (where $\delta_2$ and
  $\varepsilon_2$ are given by Lemma \ref{lem MFCQ SSOSC CRCQ})  and $\gamma>0$ such that for $x \in
  \boldsymbol{B}_{\delta_2'}(x^*)$, $y(x) =\arg\min\limits_y\{f(x,y):y\in Y(x)\cap \boldsymbol{B}_{\varepsilon_2'}(y^*) \} $
  and $(x,y(x))\in \Phi$,
\begin{equation}\label{equ-growth}
  F(x,y(x))\ge F(x^*,y^*)+\gamma \|x-x^*\|^2,\notag
\end{equation}
which indicates that $\left(x^*, y^*\right)$ is a bi-local minimum point of (BP).
\end{itemize}
\end{theorem}
\begin{proof}
Since assumptions \textnormal{(A1)}--\textnormal{(A3) hold and $(x^*,y^*)$ is a bi-local solution of (BP), we   know that
  $x^*$ is a local solution of (SP) from  Lemma \ref{lem MFCQ SSOSC CRCQ}.}
 It follows from Proposition \ref{prop-second-order tangent cone=linearized} and (\ref{composite chain rule h''})
 that
 the outer second-order tangent set $\mathcal{T}_{\Omega_1}^2(x^*,d_x)$ coincides with
 (\ref{second-order tangent set of Omega_1}).
  Since the functions involved in (SP) all are local Lipschitz continuous and second-order
 directionally differentiable at $x^*$, then from Theorem \ref{thm SONC for nonsmooth}, we can derive the
 assertion (i).

Suppose that  the condition (ii) holds. From  Lemma \ref{lem MFCQ SSOSC CRCQ},
  we know that $y(x)$ is the unique local solution of  Problem ($P_{x}$)
  and is second-order directionally differentiable and second-order gph-regular
 for $x\in \boldsymbol{B}_{\delta_2'}(x^*)$.
  By Proposition \ref{prop a composition
   second-order gph-regular}, the functions involved in (SP) all  are second-order gph-regular at $x^*$.
   Then it follows from %Theorem \ref{thm SOSC for nonsmooth}
Corollary \ref{cor SOSC for nonsmooth}   that the second-order sufficient optimality condition at $x^*$ guarantees the
   local quadratic growth condition.
We complete the proof.
\end{proof}

\begin{theorem}\label{thm SOSC for FP}
 Suppose that assumptions  \textnormal{(A2)} and \textnormal{(A4)}  hold and the MSCQ for $\Omega_2$ holds at $(x^*,y^*,\mu^*,\xi^*)$.
  \begin{itemize}
    \item[(i)]If $(x^*,y^*)$ is a bi-local solution of (BP),
 then for every $d\in
      \mathcal{C}_{\Omega_2}(x^*,y^*,\mu^*,\xi^*)$ and all $w\in \mathbb{R}^{n+m+r+s}$ satisfying
      \begin{subequations}\label{second-order tangent set of Omega_2}
      \begin{align}
      \label{Omega_2 a}
 &  d_{(x,y)}^T\nabla^2 H_i(x^*, y^*)d_{(x,y)}  +\nabla H_i(x^*, y^*)^{T}w_{(x,y)}=0,~ i\in[p], \\
 \label{Omega_2 b}
 &    d_{(x,y)}^T\nabla^2 G_i(x^*, y^*)d_{(x,y)}  +\nabla G_i(x^*, y^*)^{T}w_{(x,y)}\le0,~ i\in I_G^2,\\
 \label{Omega_2 c}
   &d^T\nabla^2 \left(\frac{\partial\mathcal{L}}{\partial y_i}\right)(x^*; y^*,\mu^*,\xi^*)d  +\nabla
   \left(\frac{\partial\mathcal{L}}{\partial y_i}\right)(x^*; y^*,\mu^*,\xi^*)^{T}w=0,~ i\in[m] ,\\
   \label{Omega_2 d}
   & d_{(x,y)}^T\nabla^2 h_i(x^*, y^*)d_{(x,y)}  +\nabla h_i(x^*, y^*)^{T}w_{(x,y)}=0,~ i\in[r], \\
& d_{(x,y)}^T\nabla^2 g_i(x^*, y^*)d_{(x,y)} +\nabla g_i(x^*,
             y^*)^{T}w_{(x,y)}-\Pi_{\mathbb{R}_-}''(g_i(x^*,y^*)+\xi_i^*;\nabla g_i(x^*, y^*)^{T}d_{(x,y)}\notag\\
\label{Omega_2 e}
&~~+d_{\xi_i},d_{(x,y)}^T\nabla^2 g_i(x^*, y^*)d_{(x,y)}+\nabla g_i(x^*, y^*)^{T}w_{(x,y)}+w_{\xi_i})=0,~ i\in [s],
      \end{align}
       \end{subequations}
      where $I_G^2:=\{i\in I_G:\nabla G_i(x^*, y^*)^{T} d_{(x,y)}=0\}$,
       it follows that
   \begin{align}%\label{F''(x^*;d_x,w_x)ge 0}
      &  d_{(x,y)}^T\nabla^2 F(x^*, y^*) d_{(x,y)}+\nabla F(x^*, y^*)^{T} w_{(x,y)}\ge 0.\notag
   \end{align}
    \item [(ii)]Assume that for every
        $d\in\mathcal{T}_{\Omega_2}(x^*,y^*,\mu^*,\xi^*)$, we have
 \begin{align*}
   & \nabla F(x^*, y^*)^{T} d_{(x,y)}\ge 0
 \end{align*} and  the optimization problem
   \begin{align}%\label{p FP F''}
    \min_{w}~~&d_{(x,y)}^T\nabla^2 F(x^*, y^*) d_{(x,y)}+\nabla F(x^*, y^*)^{T} w_{(x,y)}\notag \\
     {\rm s.t.}~~&w ~{\rm satisfying}~ (\ref{second-order tangent set of Omega_2}).\notag
   \end{align}
possesses a positive optimal objective value for every $d\in \mathcal{C}_{\Omega_2}(x^*,y^*,\mu^*,\xi^*)\backslash \{0\}$. Then there exist  $\delta_3' \in(0, \delta_3), \varepsilon_3' \in(0, \varepsilon_3)$ (where $\delta_3$ and
  $\varepsilon_3$ are given by Lemma \ref{lemma Assump A y()})
   and $\gamma'>0$ such that for $x \in
  \boldsymbol{B}_{\delta_3'}(x^*)$, $y(x) =\arg\min\limits_y\{f(x,y):y\in Y(x)\cap \boldsymbol{B}_{\varepsilon_3'}(y^*) \} $
  and $(x,y(x))\in \Phi$,
\begin{equation}\label{equ-growth}
  F(x,y(x))\ge F(x^*,y^*)+\gamma' \|x-x^*\|^2,
\end{equation}
which indicates that $\left(x^*, y^*\right)$ is a bi-local minimum point of (BP).
  \end{itemize}

Furthermore, if
the GMFCQ for $\Omega_1$ holds at $x^*$, or $\Omega_2$ holds at $(x^*,y^*,\mu^*,\xi^*)$, then the second-order conditions are equivalent to the ones in Theorem \ref{thm SOSC for SP}.
\end{theorem}

\begin{proof}
 Since assumptions  \textnormal{(A2)} and \textnormal{(A4)}  hold and  $(x^*,y^*)$ is a bi-local solution of (BP), we  know that
  $(x^*,y^*,\mu^*,\xi^*)$ is a local solution of (FP) from   \cite[Remark 3.6]{Liu 25}.
It follows from Proposition \ref{prop-second-order tangent cone=linearized} and (\ref{composite chain rule h''})
 that we derive that the outer second-order tangent set
 $\mathcal{T}_{\Omega_2}^2((x^*,y^*,\mu^*,\xi^*),d)$ coincides with (\ref{second-order tangent set of Omega_2}).
 Since the functions involved in (FP) are local Lipschitz continuous and
second-order directionally differentiable, then it follows from Theorem \ref{thm SONC for nonsmooth} that we can derive the
 assertions (i).

Suppose that the condition (ii) holds.  From Lemma \ref{lemma Assump A y()}, we know  that
$(y(x),\mu(x),\xi(x))$ is   the locally unique solution pair satisfying the KKT conditions of  Problem ($P_{x}$) for $x\in \boldsymbol{B}_{\delta_3'}(x^*)$.
By  Example \ref{example min} and  Proposition \ref{prop a composition
   second-order gph-regular},  the functions involved in (FP) are second-order gph-regular at $(x^*,y^*,\mu^*,\xi^*)$.
   Then it follows from %Theorem \ref{thm SOSC for nonsmooth}
 Corollary \ref{cor SOSC for nonsmooth}  that
 (\ref{equ-growth}) holds.

Finally, we only need to prove that the second-order conditions of Theorems \ref{thm SOSC for SP} and \ref{thm SOSC for FP} are equivalent.
It follows from the proof of Theorem \ref{thm FONC for FP} that
$d_x\in \mathcal{C}_{\Omega_1}(x^*)$ if and only if $d=(d_x,y'(x^*;d_x),\mu'(x^*;d_x),\xi'(x^*;d_x))\in
\mathcal{C}_{\Omega_2}(x^*,y^*,\mu^*,\xi^*)$, $I_G^1=I_G^2$.

From (\ref{Fkkt}), we obtain that $F_{KKT}''(x;d_x,w_x)=0$ for any $w_x$.
Then for any $w_x\in \mathcal{T}_{\Omega_1}^2(x^*,d_x)$, it follows from the chain rule in  (\ref{composite chain rule h''}) that  (\ref{Omega_2 a})-(\ref{Omega_2 e}) hold, i.e.,
 $$(w_x,y''(x^*;d_x,w_x),\mu''(x^*;d_x,w_x),\xi''(x^*;d_x,w_x))\in
  \mathcal{T}_{\Omega_2}^2((x^*,y^*,\mu^*,\xi^*),d).$$
Similarly as the proof of  Theorem \ref{thm FONC for FP}, for fixed $d\in \mathcal{C}_{\Omega_2}(x^*,y^*,\mu^*,\xi^*)$,  the system (\ref{Omega_2 c})-(\ref{Omega_2 e}) has a unique solution $(w_x,y''(x^*;d_x,w_x),\mu''(x^*;d_x,w_x),\xi''(x^*;d_x,w_x))$.

 Conversely, for any $w\in \mathcal{T}_{\Omega_2}^2((x^*,y^*,\mu^*,\xi^*),d)$, it follows from (\ref{Omega_2 a}) and (\ref{Omega_2 b}) that
 $(w_x,y''(x^*;d_x,w_x))\in \mathcal{T}_{\Omega_1}^2(x^*,d_x)$.
  It
is straightforward to derive that the second-order conditions are equivalent.
The proof is completed.
\end{proof}

Under the MFCQ, SSOSC, and CRCQ, Theorem \ref{thm SOSC for SP} establishes second-order conditions for bi-local solutions, which depend only on the first- and second-order derivatives of the problem data together with the lower-level solution mapping. When LICQ holds, these conditions become fully explicit. In contrast, the second-order condition in \cite[Theorem 5.3]{Mehlitz 21} applies to classical local solutions and is formulated through the second-order epi-regularity of $-\varphi$, where $\varphi$ denotes the optimal value function of the lower-level problem.
This condition requires lower-level convexity, LICQ, and the SOSC \cite[Proposition 5.10]{Mehlitz 21}.
Furthermore, their condition requires the computation of second-order directional derivatives of the value function $\varphi$, which is typically intricate in bilevel optimization.
 Notably, when the lower-level problem is convex, our notion of bi-local solution reduces to a local solution, so our conditions provide a more explicit and verifiable alternative to those in \cite{Mehlitz 21}.

%We now compare our results with those obtained by Mehlitz and Zemkoho in \cite[Section 5]{Mehlitz 21}. The second-order sufficient optimality condition established in \cite[Theorem 5.3]{Mehlitz 21} is based on the second-order epi-regularity of the function $-\varphi$, where $\varphi$ denotes the optimal value
%function of the lower-level problem.  In particular,  \cite[Proposition 5.10]{Mehlitz 21} shows that the convexity of the lower-level problem, together with the SOSC and LICQ
%  ensures $\varphi$ and $-\varphi$ are second-order epi-regular. This, in turn,  implies that
%  $\varphi$ is second-order gph-regular, as defined in Definition \ref{def Second-order gph-regular}.
%  In contrast, our work establishes  second-order sufficient optimality conditions for (BP) by analyzing the single-level reformulation (SP), without requiring the uniqueness of the lower-level multiplier, as shown in Theorem \ref{thm SOSC for SP}. In our setting, the MFCQ, SSOSC and CRCQ ensure the uniqueness of the global solution to the convex lower-level problem, which is also second-order gph-regular. Therefore our second-order sufficient condition in Theorem \ref{thm SOSC for SP} is derived without relying on LICQ at the lower-level.

 We conclude this section with an illustrative example.
 \begin{example}
   Consider the following bilevel problem:
   \begin{eqnarray*}
\min\limits_{x,y} && F(x,y):=x^2+(y-1)^2\\
{\rm s.t.} && G(x,y):=1-2x-y\le0,\\
&& y\in \arg\min\limits_{y}\ f(x,y):=(y-2)^2\\
&&~~~~~~~~~~~{\rm s.t.}\ \
g(x,y) :=\left(\begin{array}{c}
                x+y-1 \\
                -x+y-1
              \end{array}\right)\leq 0.
              \end{eqnarray*}
 \end{example}
 The global solution function and the corresponding set of Lagrange multipliers for the lower-level problem are given by
 \begin{align*}
    & y(x)=\left\{\begin{array}{ll}
                                        1+x, &x\le0,  \\
                                        1-x,&x>0,
                                       \end{array}\right.
                                       ~\text{and}~
                                        \Lambda_{x}(y(x))=
                                       \left\{\begin{array}{ll}
                                        (0,2(1-x)), &x<0,  \\
                                        co\{(0,2),(2,0)\}, &x=0,\\
                                        (2(1+x),0),&x>0,
                                       \end{array}\right.
 \end{align*}
 respectively.
 Consider the point $(x^*,y^*)=(0,1)$. The lower-level problem $(P_{x^*})$ satisfies the SSOSC, MFCQ and CRCQ, but not LICQ at $y^*$. Recall that
 \begin{align*}
     \Omega_1=\{x\in\mathbb{R}:G(x,y(x))\le0\}&=\{x\in\mathbb{R}:1-2x-y(x)\le0\}\\
     &=\{x\le0:1-2x-(1+x)\le0\}\cup\{x>0:1-2x-(1-x)\le0\} \\
     &=\{x\ge 0\},
 \end{align*}
    so that  the tangent cone at $x^*$ is  $T_{\Omega_1}(x^*)=\{d_x\in\mathbb{R}:d_x\ge0\}$.
 Let $\widehat{F}(x):=F(x,y(x))$. Then
 the directional derivative at $x^*$ is
  $$\widehat{F}'(x^*;d_x)=\nabla_x F(x^*, y^*)^{T} d_x+\nabla_y F(x^*, y^*)^{T}y'(x^*;d_x)=0.$$
  Therefore, for every $d_x\in T_{\Omega_1}(x^*)$, it holds that $\widehat{F}'(x^*;d_x)\ge 0$. The   critical cone  at $x^*$ for implicit   reformulation (SP) is
 \begin{align*}
  & \mathcal{C}_{\Omega_1}(x^*)=  \{d_{x}\in\mathcal{T}_{\Omega_1}(x^*):\widehat{F}'(x^*;d_x)\leq 0\}=\{d_x\ge 0\}.
\end{align*}
Hence for every $d_x\in \mathcal{C}_{\Omega_1}(x^*)\backslash \{0\}$, i.e. $d_x>0$, we obtain  from (\ref{C2 g'= g''=}) and (\ref{composite chain rule h''})  that
  \begin{align*}
   & \widehat{F}''(x^*;d_x,w_x)\\
    =&\nabla F(x^*, y^*)^{T}\left( \begin{array}{c}
                                                         w_x \\
                                                         y''(x^*;d_x,w_x)
                                                       \end{array}\right)+
                                                       \left( \begin{array}{c}
                                                         d_x \\
                                                         y'(x^*;d_x)
                                                       \end{array}\right)^T\nabla^2 F(x^*, y^*)\left( \begin{array}{c}
                                                         d_x \\
                                                         y'(x^*;d_x)
                                                       \end{array}\right)\\
    =&2(d_x)^2+2(y'(x^*;d_x))^2>0.
 \end{align*}
 Consequently, the point $(x^*,y^*)$ is a strict local minimizer of the given bilevel optimization problem  by
 Theorem \ref{thm SOSC for SP}.

\section{Conclusions}\label{section 5}
Based on the parabolic curve approach, we establish comprehensive no-gap second-order necessary and sufficient optimality conditions for constrained nonsmooth optimization by introducing the novel concept of second-order gph-regularity and deriving the outer second-order regularity of the feasible region, without imposing convexity assumptions on the constraint set.
In the second part, we show that the local solution mapping of a parametric problem satisfies second-order gph-regularity under the MFCQ, SSOSC, and CRCQ.
Moreover, under the SSOSC and LICQ, these conditions are equivalent to second-order optimality conditions that involve only the directional derivatives of the defining functions of the bilevel problem.

\end{document}